\documentclass[10pt,reqno]{amsart}

\allowdisplaybreaks

 \numberwithin{equation}{section}

\newtheorem{theorem}{Theorem}[section]
\newtheorem{lemma}[theorem]{Lemma}
\newtheorem{corollary}[theorem]{Corollary}

\theoremstyle{remark}
\newtheorem{remark}[theorem]{Remark}

\theoremstyle{definition}
\newtheorem{assumption}[theorem]{Assumption}
\newtheorem{example}[theorem]{Example}
\newtheorem{definition}[theorem]{Definition}
\newtheorem{proposition}[theorem]{Proposition}

\def\XXint#1#2#3{{\setbox0=\hbox{$#1{#2#3}{\int}$}
\vcenter{\hbox{$#2#3$}}\kern-.5\wd0}}

\newcommand\cbrk{\text{$]$\kern-.15em$]$}}
\newcommand\opar{\text{\,\raise.2ex\hbox{${\scriptstyle
|}$}\kern-.34em$($}}
\newcommand\cpar{\text{$)$\kern-.34em\raise.2ex\hbox{${\scriptstyle |}$}}\,}

\def\<{\langle}
\def\>{\rangle}
\def\eps{\varepsilon}
\def\E{{\mathbb E}}

\newcommand\bL{\mathbb{L}}
\newcommand\bR{\mathbb{R}}
\newcommand\bC{\mathbb{C}}
\newcommand\bH{\mathbb{H}}

\newcommand\bD{\mathbb{D}}
\newcommand\bS{\mathbb{S}}

\newcommand\bE{\mathbb{E}}

\newcommand\bP{\mathbb{P}}

\newcommand\cB{\mathcal{B}}

\newcommand\cF{\mathcal{F}}

\newcommand\cH{\mathcal{H}}

\newcommand\cP{\mathcal{P}}
\newcommand\cR{\mathcal{R}}

\def\R {{\mathbb R}}

\begin{document}

\title[Fractional Time SPDEs]
{Fractional time stochastic partial differential equations}

\author{{\bf Zhen-Qing Chen} }
\address{Department of Mathematics, University of Washington,
Seattle, WA 98195, USA}
\email{zqchen@uw.edu}
\thanks{Research of ZC is partially supported by NSF Grants NSF Grant DMS-1206276.}

\author{{\bf Kyeong-Hun Kim}  }
\address{Department of Mathematics, Korea University, 1 Anam-dong,
Sungbuk-gu, Seoul, 136-701, Republic of Korea}
\email{kyeonghun@korea.ac.kr}
\thanks{Research of KK was supported by Basic Science Research Program through the
National Research Foundation of Korea(NRF) funded by the Ministry of
Education, Science and Technology (2013020522)}

\author{{\bf Panki Kim}}
\address{Department of Mathematical Sciences and Research Institute of Mathematics,
Seoul National University,
San56-1 Shinrim-dong Kwanak-gu,
Seoul 151-747, Republic of Korea}
\email{pkim@snu.ac.kr}
\thanks{Research of PK was supported by the National Research Foundation of Korea(NRF) grant funded by the Korea government(MEST) (2013004822).}



\begin{abstract}
In this paper, we introduce a class of stochastic partial differential equations
(SPDEs) with fractional time-derivatives, and study the $L_2$-theory of the equations.
This class of SPDEs can be used to describe random effects on transport of particles in medium with
thermal memory or particles subject to sticking and trapping.
\end{abstract}

\maketitle

\section{\bf Introduction}

Fractional calculus has attracted lots of attentions in several fields including mathematics, physics, chemistry, engineering, hydrology
 and even finance and social sciences.
The classical heat equation $\partial_t u=\Delta u$ describes heat propagation in homogeneous medium.
The time-fractional diffusion equation $\partial^\beta_t u=\Delta  u$
with $0<\beta<1$  has been widely used
to model the anomalous diffusions  exhibiting  subdiffusive behavior,
due to particle sticking and trapping phenomena.
Here  the fractional time
derivative $\partial^{\beta}_t$ is the Caputo derivative of order $\beta\in (0,1)$, defined by
\begin{equation}\label{e:1.1}
\frac{\partial^\beta f(t)}{\partial t^\beta}=\frac{1}{\Gamma (1-\beta)} \frac{d}{dt}
\int_0^t (t-s)^{-\beta}  \left(f(s)-f(0)\right) ds
\end{equation}
where $\Gamma$ is the Gamma function defined by $\Gamma(\lambda):=
\int^{\infty}_0 t^{\lambda-1} e^{-t}dt$.

 Fractional diffusion equations are becoming popular in many areas of application \cite{GM, KS,  LS,MK,  P, SBMW, V}.
  So far, on the basis of either deterministic or probabilistic methods,
   the study of fractional calculus is mainly  restricted to deterministic equations; see \cite{EIK, MS, P, Za} and the references therein.
 In this paper, we introduce  and investigate a class of stochastic partial differential equations (SPDEs) with fractional time derivatives.

The SPDEs  with fractional time derivative
that we are going to study in this paper
 naturally arise from  the consideration of the heat equation in
a material of thermal memory. Let $u(t,x), e(t,x)$ and $\vec F(t,x)$ denote
the body temperature, internal energy and flux density, respectively.
Then the relations
\begin{equation}
  \label{heat 1}
\frac{\partial e}{\partial t}(t,x)=-{\rm div} \, \vec  F,
\end{equation}
$$
e(t,x)=\beta u(t,x), \quad \vec F(t,x)=-\lambda \nabla u(t,x) , \quad\quad
\beta, \lambda >0
$$
 yield the classical heat equation
$\beta \frac{\partial u}{\partial t}=\lambda \Delta u$. According to
the law of the classical heat equation, the speed of the heat flow
is infinite. However in real modeling, the propagation speed can be
finite because the heat flow can be disrupted by the response of the
material. It has been proved
(see e.g., \cite{V,LS})
that in a material with thermal memory
\begin{equation}
             \label{2014.5}
e(t,x)=\bar{\beta} u(t,x)+\int^t_0 n(t-s)u(s,x)ds
\end{equation}
holds with some appropriate constant $\bar{\beta}$ and kernel  $n$.
Typically, $n(t)$ is a positive decreasing function
which blows up near $t=0$, for instance $n(t)=t^{-\alpha}$ for $\alpha\in (0,1)$.  In
this case  the convolution implies that  nearer past affects the present
more. If in addition the internal energy $e(t,x)$ depends also on
past random effects, then \eqref{2014.5} becomes
\begin{equation}
             \label{2014.6}
e(t,x)=\bar{\beta} u(t,x)+\int^t_0 n(t-s)u(s,x)ds+\int^t_0
\ell(t-s)h(s, u(s,x))dW_s,
\end{equation}
where $W$ is a random process, such as Brownian motion, modeling
the random effects.
If $u(0, x)=0$, $\bar{\beta}=0$, $n(t)=\Gamma ( 1-\beta_1 )^{-1} t^{-\beta_1}$ and $\ell(t)=\Gamma (2- \beta_2 )^{-1}t^{1-\beta_2}$
for some constants $\beta_i \in (0,1)$,  then
 (\ref{2014.6}) (after differentiation in $t$) becomes
\begin{eqnarray}\label{e:1.5}
-{\rm div} \, \vec  F  = \frac{\partial e}{\partial t}(t,x)&=&\frac{1}{\Gamma (1- \beta_1 )}\frac{\partial }{\partial t} \int^t_0  (t-s)^{-\beta_1}    u(s,x)ds
\nonumber \\
&&+\frac1{\Gamma (2- \beta_2 )}\frac{\partial }{\partial t}\int^t_0
(t-s)^{1-\beta_2}h(s, u(s,x))dW_s.
\end{eqnarray}
Since
\begin{align*}
\int_0^t (t-s)^{-{\beta_2}} \int_0^s
h(a, u(a,x))
dW_ads&= \int_0^t
 \int_a^t  (t-s)^{-{\beta_2}}ds
h(a, u(a,x))
dW_a
\\&=\frac{1}{1-{\beta_2}} \int_0^t
 (t-a)^{1-{\beta_2}}
h(a, u(a,x))
dW_a,
\end{align*}
 by the definition of Caputo derivative (\ref{heat 1}) we have
\begin{align*}
\partial^{\beta_2}_t \int_0^t
h(s, u(s,x))
dW_s&=
\frac{1}{\Gamma (1-{\beta_2})}\frac{\partial }{\partial t}
\int_0^t (t-s)^{-{\beta_2}} \int_0^s
h(a, u(a,x))
dW_ads\\
&=
\frac{1}{\Gamma (2-{\beta_2})}\frac{\partial }{\partial t}
 \int_0^t
 (t-s)^{1-{\beta_2}}
h(s, u(s,x))
dW_s.
\end{align*}
Thus \eqref{e:1.5} can be rewritten as
the following stochastic partial differential equation
involving fractional time-derivative
\begin{equation}\label{e:1.4}
\partial^{\beta_1}_t u = {\rm div}  \, \vec F + \partial^{\beta_2}_t \int_0^t
h(s, u(s,x)) dW_s.
\end{equation}
It is this type of stochastic equations and its natural extensions
that will be studied in this paper.

Now let $(\Omega,\cF,\bP)$ be a complete probability space,
$\{\cF_{t},t\geq0\}$ be an increasing filtration of $\sigma$-fields
$\cF_{t}\subset\cF$, each of which contains all $(\cF,\bP)$-null sets.
We assume that on $\Omega$ we are given an independent family of one-dimensional
Wiener processes
 $W^{1}_{t}$, $W^{2}_{t},...$
relative to the filtration $\{\cF_{t},t\geq0\}$.

Motivated by \eqref{e:1.4},
 in this paper we consider the following quasi-linear
 SPDEs   of the non-divergence form type
\begin{eqnarray}
    \nonumber
  \partial^{\beta}_t u&=&\left(a^{ij}u_{x^ix^j}+b^iu_{x^i}+cu+f(u) \right)\\
   && +\sum_{k=1}^\infty\partial^{\gamma}_t \int^t_0 (\sigma^{ijk}u_{x^ix^j} +\mu^{ik}u_{x^i}+\nu^ku+g^k(u))\, d
   W^k_s
    \label{eqn 2014.1}
 \end{eqnarray}
 as well as of the divergence form type
 \begin{eqnarray}
                  \nonumber
  \partial^{\beta}_t u &=&\left( D_i\left(
a^{ij}u_{x^j}+b^iu+f^i(u)\right)+cu+h(u)
  \right) \\
   && + \sum_{k=1}^\infty\partial^{\gamma}_t \int^t_0 (\sigma^{ijk}u_{x^ix^j} +\mu^{ik}u_{x^i}+\nu^ku+g^k(u))\,
    dW^k_s,                  \label{2014.2}
 \end{eqnarray}
given for $\omega\in \Omega, \, t\geq0$  and $x\in \bR^d$, and study the $L_2$-theory
of the equations.
The constants  $\beta,\gamma\in (0,1)$ are assumed to satisfy the condition
\begin{equation}\label{e:1.7}
\gamma<\beta+1/2.
\end{equation}
The indices $i$ and $j$ go from $1$ to the dimension $d$  with the summation convention on $i,j$ being
enforced.
The coefficients $a^{ij}, b^i, c, \sigma^{ijk}, \mu^{ik}, \nu^k$  are functions depending on
$(\omega,t,x)$ and the functions $f,f^i, h, g^k$ depend on $(\omega,t,x)$ and the unknown $u$.  By considering infinitely many independent Brownian motions $W^k_t$ we cover equations driven by measure-valued processes, for instance, driven by space-time white noise (see Section 3.3). It is worth mentioning that unlike the classical SPDE, we allow the second-order derivatives of the unknown solution $u$ to  appear in the stochastic part when $\gamma<1/2$.

As for  stochastic differential equations (SDEs),
SPDE \eqref{eqn 2014.1} should be interpreted by its integral form
\begin{eqnarray*}
&&u(t,x)-u(0,x)\\
&=&\frac{1}{\Gamma(\beta)}\int^t_0
(t-s)^{\beta-1}\left(a^{ij}(s,x)u_{x^ix^j}(s,x)+\cdots+f(s,
u(s,x))\right)ds\\
&&+\frac{1}{\Gamma(1+\beta-\gamma)}
\int^t_0(t-s)^{\beta-\gamma}\left(\sigma^{ijk}(s,x)u_{x^ix^j}(s,x)+\cdots+g^k(s,
u(s,x))\right)dW^k_s.
\end{eqnarray*}
Similarly one can write down the integral version of SPDE \eqref{2014.2} but
in the distributional sense with respect to $x$ variable.

We next explain the constraint \eqref{e:1.7}.
A special case of the SPDEs for both \eqref{eqn 2014.1} and \eqref{2014.2}
is
\begin{equation}\label{e:1.8}
\partial^\beta_t u (t,x) =\Delta u(t,x) + \partial^\gamma_t \int_0^t g(s,x) dW_s,
\end{equation}
where $W$ is a one-dimensional Brownian motion.
For functions $h_1$ and $h_2$ on $\R_+$, we define its convolution
$h_1*h_2$ by
$$
h_1*h_2 (t)=\int^t_0 h_1(t-s)h_2(s)ds.
$$
Let
$$
k_{\beta}(t):=\Gamma(\beta)^{-1} t^{\beta-1},
$$
 and define
$$I^{\beta} \varphi =k_{\beta}*\varphi := \int_0^t k_\beta (t-s)\varphi (s) ds.
$$
One can easily check for any $\beta, \gamma \in (0,1)$,
$k_\beta * k_\gamma = k_{\beta+\gamma}$.
So in particular, we have
\begin{equation}   \label{e:1.10}
k_{\beta}*k_{1-\beta} (t)\equiv 1 .
\end{equation}
Suppose $u(x, t)$ is a solution of (\ref{e:1.8}).
In view of the definition
of Caputo derivative \eqref{e:1.1},
equation \eqref{e:1.8} is understood by its integral version
$$
k_{1-\beta}*(u(t,x)-u(0, x))=\int^t_0 \Delta u(s,x)ds+k_{1-\gamma}*\int^t_0
g(s)dW_s.
$$
By taking convolution with $k_{\beta}$ on both sides,   we
  get from \eqref{e:1.8} and \eqref{e:1.10}
$$\int^t_0
\left(u(s,x)-u(0,x)\right)\,ds=\left(I^{\beta}*\int^\cdot_0 \Delta
u(s,x)ds\right)(t)+\left(I^{\beta+1-\gamma}* \int^\cdot_0 g(s,x)dW_s\right)(t).
$$
Since the first two terms are differentiable in $t$, the last term above
should be differentiable in $t$.
Recall that
$$
I^{a}: C^{b}\to C^{a+b},
$$
and $\int^t_0 g(s)dW_s \in C^{1/2-\varepsilon}$ for any
$\varepsilon>0$. Thus we must have
$$
\beta+1-\gamma>1/2,
$$
which is equivalent to \eqref{e:1.7}.

\medskip

The main results of this paper are Theorems \ref{thm divtype} and \ref{main},
on the unique solvability of SPDEs \eqref{eqn 2014.1} and \eqref{2014.2}, and $L_2$-estimates of their solutions.
For SPDE \eqref{eqn 2014.1},  we establish in Theorem \ref{main}
the unique solvability in the space $L_2(\Omega\times[0,T], H^{\sigma}_2)$ for any $\sigma\in \bR$ under appropriate differentiability assumption on $x$-variable of the coefficients. On the other hand, the unique solvability of SPDE \eqref{2014.2} in the space $L_2(\Omega\times[0,T], H^{1}_2)$ is obtained in Theorem \ref{thm divtype} under the merely measurability condition of the coefficients $a^{ij}$.

The rest of the paper is organized as follows.
In Section 2 we present some preliminary results on the fractional derivatives and in Section 3 we introduce stochastic Banach spaces and few key estimates. Our main results for  (\ref{2014.2}) and (\ref{eqn 2014.1})   are presented and proved in Section 4 and Subsection 5.1, respectively.
Subsection 5.2 contains an application to an equation driven by space-time white noise.

We close this section with some notation. As usual, $\bR^{d}$
stands for the Euclidean space of points $x=(x^{1},...,x^{d})$.
 For $i=1,...,d$, multi-indices $\alpha=(\alpha_{1},...,\alpha_{d})$,
$\alpha_{i}\in\{0,1,2,...\}$, and functions $u(x)$ we set
$$
u_{x^{i}}=\partial u/\partial x^{i}=D_{i}u,\quad
D^{\alpha}u=D_{1}^{\alpha_{1}}\cdot...\cdot D^{\alpha_{d}}_{d}u,
\quad|\alpha|=\alpha_{1}+...+\alpha_{d}.
$$
We also use the notation $D^m_x$ for a partial derivative of order $m$
with respect to $x$. By  $C^{\infty}_0(\bR^d)$ we denote  the collection of
smooth
functions having compact support in $\bR^d$. For $p\geq 1$, let
$$
L_p=L_p(\bR^d):=\{ u : \bR^d \to \bR,  \| u \|^p_{L_p}:=\int_{\bR^d} |u(x)|^p dx < \infty   \}
$$
and we use the notation $(f,g)_{L_2}:= \int_{\bR^d} f(x)g(x) dx$.
We denote
$$\cF(g)(\xi):=\frac{1}{(2\pi)^{d/2}}\int_{\bR^d}e^{-i\xi\cdot
x}g(x)dx \quad \text{and}\quad \cF^{-1}(f)(\xi):=\frac{1}{(2\pi)^{d/2}}\int_{\bR^d}e^{i\xi\cdot
x}f(x)dx$$
  the Fourier transform of $g$ in $\bR^d$ and
  the  inverse Fourier transform of $f$ in $\bR^d$, respectively.

If we write $N=N(a,b,\cdots)$, this means that the
constant $N$ depends only on $a,b,\cdots$.  Throughout
this paper, for functions depending on $(\omega,t,x)$, usually the argument
$\omega \in \Omega$ will  be omitted.

\section{\bf Preliminary results}
                         \label{section 1}
First we introduce a few elementary facts on the fractional derivatives.
The reader
can find further details in \cite{EIK} and references therein.
Recall that
$\beta\in (0,1)$ and
$$
k_{\beta}(t):=t^{\beta-1}\Gamma(\beta)^{-1}, \quad t>0.
$$
Let $T>0$. If $f$ is absolutely continuous  on $[0, T]$ with $f(0)=0$ then
\begin{equation}
                         \label{2014.1.23.2}
\frac{d}{dt} (k_{\beta}*f)=k_{\beta}* \frac{d}{dt}f, \quad t \in [0, T].
\end{equation}
For functions $\varphi\in L_1([0,T])$,  the Riemann-Liouville fractional integral of the order $\beta\in (0,1)$ is defined by
$$
I^{\beta} \varphi (t)=k_{\beta}*\varphi (t)=\frac{1}{\Gamma(\beta)}\int^t_0 (t-s)^{\beta-1}\varphi(s)ds.
$$
Note that by Jensen's inequality
$$ |I^\beta \varphi (t)|^p \leq \frac{1}{(\Gamma (\beta))^p} \left(\frac{t^\beta}{\beta}\right)^{p-1}
\int_0^t (t-s)^{\beta-1} |\varphi (s)|^p ds.
$$
Thus it follows that
 for any $p\in [1,\infty]$,
\begin{equation}
                        \label{eqn 7.03.1}
\|I^{\beta}\varphi\|_{L_p([0,T])}\leq N(T,\beta) \|\varphi\|_{L_p([0,T])}.
\end{equation}
 Consequently, if $\varphi_n \to \varphi$  in $L_p([0,T])$ then $I^{\beta}\varphi_n$ also converges  to $I^{\beta} \varphi$ in $L_p([0,T])$.
Also one can prove that if $f_n(\omega,t)$ converges in probability uniformly in $[0,T]$ then so does $I^{\beta}f_n$.

If $I^{1-\beta}\varphi$ is absolutely continuous, then  Riemann-Liouville derivative of order $\beta$ is defined by
\begin{equation}
          \label{Riemann-Liouville}
D^{\beta}_t\varphi (t)=\frac{d}{dt} (I^{1-\beta}\varphi)(t).
\end{equation}
If $\varphi$ is continuous and $I^{1-\beta}\varphi$ is absolutely continuous, the generalized functional derivative (or the Caputo derivative)  of order $\beta$
is given by
\begin{eqnarray}
          \label{caputo}
\partial^{\beta}_t\varphi (t)&=&D^{\beta}_t(\varphi-\varphi(0))
= D^{\beta}_t\varphi (t)-\frac{\varphi(0)}{t^{\beta}\Gamma(1-\beta)}.
\end{eqnarray}
It is easy to check for any $\varphi\in L_1([0,T])$,
\begin{eqnarray}
          \label{e:DI}
D^{\beta}_tI^{\beta}\varphi =\varphi.
\end{eqnarray}
Furthermore, if
 $\varphi$ is absolutely continuous on $[0,T]$ then by \eqref{2014.1.23.2}
\begin{eqnarray} \label{e:pbetav}
\partial^{\beta}_t\varphi=I^{1-\beta}\frac{d}{dt}\varphi.
\end{eqnarray}
Thus by \eqref{e:DI}
\begin{eqnarray} \label{e:Dbetadt}
D^{1-\beta}_t\partial^{\beta}_t \varphi=\frac{d}{dt}\varphi, \quad \text{a.e.}.
\end{eqnarray}

Denote by
$$
E_{\beta}(z)=\sum_{k=0}^{\infty}\frac{z^k}{\Gamma(\beta k+1)}, \quad z \in \bC
$$
the Mittag-Leffler function.
We will also use the generalized Mittag-Leffler function
$$
E_{\beta,\gamma}(z):=\sum_{k=0}^{\infty}\frac{z^k}{\Gamma(\beta k+\gamma)}, \quad z \in \bC.
$$

We assume $\beta, \gamma \in (0,1)$.
It is well known (see e.g. (12) in  \cite[Theorem 1.3-4]{D}) that,
when $-1 < \gamma -\beta < 1$,
 $E_{\beta,\gamma}(t)$ is bounded on $(-\infty,0]$ and
\begin{equation}
          \label{mittag}
\lim_{\bR \ni t \to \infty}  t E_{\beta,\gamma}(-t) =\frac{1}{\Gamma (\gamma-\beta)}.
\end{equation}
Furthermore, for any constant $\lambda$,
\begin{equation}
              \label{frac heat}
\varphi(t):=E_{\beta}(\lambda t^{\beta})
\end{equation}
is a solution of the equation
$$
\partial^{\beta}_t \varphi=\lambda \varphi, \quad t>0
$$
with the initial condition $\varphi(0)=1$.

The following is a classical result. We provide a proof  for the readers' convenience.

\begin{lemma}\label{l:DaE}
Let $a\in (0,1)$ and $b \ge 0$,
 then
\begin{equation}
                 \label{seattle1}
 D_t^a E_{\beta}(-b
t^{\beta})=t^{-a}E_{\beta,1-a}(-bt^{\beta}),
\end{equation}
\begin{equation}
               \label{seattle2}
I^aE_{\beta}(-b
t^{\beta})=t^{a}E_{\beta,1+a}(-bt^{\beta}).
\end{equation}

\end{lemma}

\begin{proof}
We first prove (\ref{seattle1}). One can easily check   (see e.g.
\cite[(5.1.2)]{EIK})
\begin{equation}
         \label{seattle3}
D_t^at^{b-1}=\frac{\Gamma(b)}{\Gamma(b-a)}t^{b-a-1}, \quad a,b>0.
\end{equation}
Thus,
\begin{eqnarray*}
D_t^a E_{\beta}(-b t^{\beta})&=&D_t^a\left(\sum_{k=0}^\infty
\frac{(-1)^kb^{2k}t^{\beta k}}{\Gamma(\beta k+1)}\right)\\
&=&\sum_{k=0}^\infty\frac{(-1)^kb^{2k}}{\Gamma(\beta
k+1)}\frac{\Gamma(\beta k+1)}{\Gamma(\beta k+1-a)} t^{\beta
k-a}\\
&=&t^{-a}\sum_{k=0}^\infty\frac{(-bt^{\beta})^k}{\Gamma(\beta
k+1-a)}\\
&=&t^{-a}E_{\beta,1-a}(-bt^{\beta}).
\end{eqnarray*}

To prove (\ref{seattle2}), in place of (\ref{seattle3}), it is
enough to use
$$
I^at^{b-1}=\frac{\Gamma(b)}{\Gamma(b+a)}t^{b+a-1}, \quad a,b>0.
$$
 The lemma is proved.
\end{proof}

Define
$$
p(t,x)=\cF^{-1} (E_{\beta}(-|\xi|^{2}t^{\beta})), \quad q(t,x)=D^{1-\beta}_t p.
$$
Actually (\ref{mittag}) shows that if $d>1$ then $E_{\beta}(-|\xi|^2t^{\beta})\not\in L_1(\bR^d)$  for fixed $t>0$. Thus we understand $p(t,x)$ as the inverse transform   of a radial function  in the sense of  improper integral, or we can define $p(t,x)$ first as in \cite[Section 5.2.2]{EIK}  so that $p(t,\cdot)\in L_1(\bR^d)$ and $\cF(p(t,x))(\xi)=E_{\beta}(-|\xi|^2t^{\beta})$.

Since, for $x\neq 0$, $p(t,x)\to 0$ as $t\downarrow 0$ (see \cite{EIK}), the Riemann-Liouville derivative of
$p(\cdot, x)$
 coincides with the Caputo derivative of
$p(\cdot , x)$ for every $x\not= 0$.

\vspace{2mm}

Here is a list of other useful properties of $p$ and $q$.

\begin{lemma}
          \label{lemma list}
(i)\, For $(t,x)\in (0,T]\times \bR^d \setminus \{0\}$, we have $\partial^{\beta}_t p=D^{\beta}_t p=\Delta p$.

(ii) For each $x\neq 0$ and $m\leq 3$,
\begin{equation}
                    \label{eqn 9.21.1}
\lim_{t\to 0+} D^m_x p(t,x)=\lim_{t\to 0+}D^m_x q(t,x)=0.
\end{equation}

 (iii)\,  $D^m_x p(t,\cdot)$ is  integrable for each  $t>0$ if $m\leq 1$.

(iv) \, For each  $t>0$ and $m\leq 3$, $D^m_x q(t,\cdot)$ is integrable in $\bR^d$ uniformly on $[\varepsilon,T]$ for any $\varepsilon>0$.

(v) \, For each $x\neq 0$, $p(\cdot,x)$ is absolutely continuous and $\frac{
\partial}{ \partial t}p(t,x)\to 0$ as $t\downarrow 0$.
Moreover,
$\frac{\partial}{ \partial t}p(t,\cdot)$ is integrable in $\bR^d$  uniformly on $[\varepsilon,T]$ for any $\varepsilon>0$.

(vi)\, For any compact set $K\subset \bR^d\setminus \{0\}$ and $m\leq 3$, the functions $p, q, \frac{\partial}{ \partial t}p(t,\cdot), D^m p$ and $D^m q$ are continuous and bounded on $[0,T]\times K$.
\end{lemma}
\begin{proof}
(i) is a consequence of (\ref{frac heat}). See (5.2.44) and (A.19) of \cite{EIK} for (v). Others can be found in Propositions 5,1 and 5.2 of \cite{EIK}. In particular, (vi) is a consequence of (5.2.6) and (5.2.13) of \cite{EIK}.
\end{proof}

We need the following integration by parts formula.

\begin{lemma}\label{L:2.6}
Suppose that  $f$ is absolutely continuous  on $[0, T]$ with $f(T)=0$ and  $g$ is continuous  and $I^{1-\beta}g$  is absolutely continuous on $[0, T]$.
Then
$$
\int_0^T f(t)D^\beta_t g(t) dt = \int_0^T G(t) D^\beta_t F(t) dt,
$$
where $F(t)=f(T-t)$ and $G(t)=g(T-t)$.
\end{lemma}

\begin{proof}
Note that
\begin{eqnarray*}
 \int_0^T |f'(t)| \left( \int_0^t \theta^{-\beta}
 |g(t-\theta )| d\theta\right) dt  \le \|g\|_{L_\infty([0,T])}
 \int_0^T |f'(t)| dt  \int_0^T \theta^{-\beta}
d\theta <\infty.
\end{eqnarray*}
Thus,
 using $f(T)=0$, the integration by part and the Fubini's Theorem, we get
\begin{eqnarray}
\int_0^T f(t) D^\beta_t g(t) dt
&=& \int_0^T f(t) \frac{d}{dt} (I^{1-\beta}  g) (t) dt  \nonumber\\
&=& f(t) (I^{ 1-\beta }  g) (t) \Big|_0^T -\int_0^T f'(t) (I^{1-\beta } g)(t) dt  \nonumber\\
&=& - \frac{1}{\Gamma (1-\beta)} \int_0^T f'(t) \left( \int_0^t \theta^{-\beta}
 g(t-\theta ) d\theta\right) dt \nonumber\\
 &=& -  \frac{1}{\Gamma (1-\beta)} \int_0^T \theta^{-\beta}   \left( \int_\theta^T
f'(t)  g(t-\theta ) dt \right) d\theta \label{e:sd} .
\end{eqnarray}
 As $F(0)=f(T)=0$ and $f$ is absolutely continuous, by \eqref{caputo} and \eqref{e:pbetav} we have
 $D^\beta_t F = \partial^\beta_t F=(I^{1-\beta} F')$. So by the integration by part and the Fubini's Theorem
  \begin{eqnarray*}
\int_0^TG(t)  D^\beta_t F(t) dt
&=&  \int_0^T G(t) (I^{1-\beta} F') (t) dt \\
&=& \frac{1}{\Gamma (1-\beta )} \int_0^T G(t) \left( \int_0^t \theta^{-\beta}
F'(t-\theta ) d\theta \right) dt \\
&=&  \frac{1}{\Gamma (1-\beta )} \int_0^T \theta^{-\beta}  \left( \int_\theta^T G(t)
F'(t-\theta ) dt\right) d\theta  \\
&=& - \frac{1}{\Gamma (1-\beta )} \int_0^T \theta^{-\beta}  \left( \int_\theta^T
f' (T-t+\theta ) g(T-t) dt\right) d\theta  \\
&=& - \frac{1}{\Gamma (1-\beta )} \int_0^T \theta^{-\beta}  \left( \int_\theta^T
f' (s) g(s-\theta) ds\right) d\theta,
\end{eqnarray*}
which is $
\int_0^T f(s) D^\beta_t g(s) ds$ by \eqref{e:sd}.
This proves the lemma.
\end{proof}

\begin{lemma}\label{L:2.3}
For each $(t, x)\in (0,
\infty) \times \bR^d\setminus\{0\}$,
\begin{equation}
               \label{eqn lap}
\Delta q (t,x)=\frac{\partial}{\partial t}p (t,x).
\end{equation}
\end{lemma}

\begin{proof}
Since both are continuous it is enough to prove that  the equality holds for almost all $(t, x)$. Let $\phi (t)$ and $\psi (x)$ be smooth functions with compact support in $(0, T)$ and $\bR^d\setminus\{0\}$
respectively. Define $\Phi (s):=\phi (T-s)$.
Since $\phi (t)$ is smooth function with compact support in $(0, T)$,
using  Lemma \ref{lemma list}(i), Lemma  \ref{L:2.6} and \eqref{e:Dbetadt}, for every $x \in \bR^d$
\begin{eqnarray}
 \int^T_0  D^{1-\beta}_t  \Phi(s)   \Delta p(T-s,x) \,ds
&=&\int^T_0 D^{1-\beta}_t  \Phi(s)  \left(\partial^{\beta}_t p(\cdot, x)\right) (T-s) \,ds  \nonumber\\
&=&  \int^T_0   \phi (s ) \left( D^{1-\beta}_t  \partial^{\beta}_t p(\cdot, x) \right) (s) \,ds \nonumber \\
&=&\int^T_0   \phi( s )  \frac{\partial}{\partial s} p(s,x) \,ds. \label{e:fnnew}
\end{eqnarray}
Recall that  $D^2_x q$, $D^2_x p$  and $\frac{\partial}{\partial s}p$ are locally integrable in $\bR^d\setminus\{0\}$ uniformly on the support of $\phi(t)$.
By  the integration by parts in $x$, Lemma \ref{lemma list}(vi), Lemma \ref{L:2.6}, \eqref{e:fnnew} and the Fubini's Theorem,
\begin{eqnarray*}
\int_{\bR^d}\int^T_0 \phi(s)  \psi (x) \Delta q (s,x) \,dsdx  &=&\int^T_0\int_{\bR^d} \phi(s) \Delta \psi (x)  q (s,x) \,dxds\\
&=&\int_{\bR^d} \Delta \psi(x)\left( \int_0^T \phi (s) \, D^{1-\beta}_t p(s,x) \,ds \right)dx\\
&=&\int_{\bR^d} \Delta \psi (x) \left( \int^T_0 D^{1-\beta}_t  \Phi (s)   p(T-s,x) \,ds \right)dx\\
&=&\int_{\bR^d} \psi (x)  \left( \int^T_0  D^{1-\beta}_t  \Phi(s)   \Delta p(T-s,x) \,ds \right)dx \\
&=&\int_{\bR^d}\int^T_0   \phi( s ) \psi (x) \frac{\partial}{\partial s} p(s,x) \,dsdx.
\end{eqnarray*}
Since $\phi(t)$ and $\psi(x)$ are arbitrary,  the lemma is proved.
\end{proof}

\section{\bf Key Estimates}

In this section, we first define a stochastic Banach space and establish key lemmas. Then we study the $L_2$-theory of a model equation for SPDEs with fractional time-derivatives.

For  $n=0,1,2,...$, define the Banach spaces
$$
H^n_2:=H^n_2(\bR^d) =\left\{u: u, D_xu, \cdots, D_x^n u \in L_2\right\}.
$$
 In general, for $\sigma \in \bR$ define the space
$H^{\sigma}_2=H^{\sigma}_2(\bR^d)=(1-\Delta)^{-\sigma/2}L_2$  as the set of all distributions $u$ on $\bR^d$ such that
$(1-\Delta)^{\sigma/2}u\in L_2$. For $u\in H^\sigma_2$, we define
\begin{equation}   \label{e:2.22}
\|u\|_{H^{\sigma}_2}:=\|(1-\Delta)^{\sigma/2}u\|_{L_2}
:=\|\cF^{-1}[(1+|\xi|^2)^{\sigma/2}\cF(u)(\xi)]\|_{L_2}.
\end{equation}
Similarly for $\ell_2$-valued $g=(g^1, g^2, \dots)$,
\begin{equation}   \label{e:2.22n}
\|g\|_{H^{\sigma}_2(\ell_2)}:=\||(1-\Delta)^{\sigma/2}g|_{\ell_2}\|_{L_2}
:=\||\cF^{-1}[(1+|\xi|^2)^{\sigma/2}\cF(g)(\xi)]|_{\ell_2}\|_{L_2}.
\end{equation}

Let $\cP$ be the predictable $\sigma$-field and $\cP^{d\bP\times dt}$ be the completion of $\cP$ with respect to  $d\bP\times dt$.  For each $\sigma \in \bR$, define the Banach space
$$
\bH^{\sigma}_2(T):=L_2(\Omega\times [0,T],\cP, H^{\sigma}_2).
$$
That is,  $u\in \bH^{\sigma}_2(T)$ if
$u$ is an $H^{\sigma}_2$-valued
 $\cP^{d\bP\times dt}$-measurable  process defined on
$\Omega \times [0,T]$ so that
$$
\|u\|_{\bH^{\sigma}_2(T)}:= \left( \E\int^T_0\|u (t, \cdot )\|^2_{H^{\sigma}_2}\,dt\right)^{1/2}<\infty.
$$
 For an $\ell_2$-valued process
$g=(g^1,g^2,...)$, we write $g\in \bH^{\sigma}_2(T,\ell_2)$ if  $g^k\in \bH^{\sigma}_2(T)$
for every $k\geq 1$ and
$$
\|g\|_{\bH^{\sigma}_2(T,\ell_2)}:=
\left(  \E \int^T_0\|g\|^2_{H^{\sigma}_2(\ell_2)}\,dt \right)^{1/2}<\infty.
$$
Denote $\bL_2(T)=\bH^0_2(T)$ and $\bL_2(T,\ell_2)=\bH^0_2(T,\ell_2)$.  Write $g\in \bH^{\infty}_0(T,\ell_2)$ if $g^k=0$ for all sufficiently large $k$,  and each $g^k$ is of the type
$$
g^k=\sum_{i=1}^nI_{(\tau_{i-1},\tau_i]}(t) g^{ik}(x)
$$
where $\tau_i$ are bounded stopping times with respect to $\cF_t$ and $g^{ik} \in C^{\infty}_0 (\bR^d)$.
It is known (\cite[Theorem 3.10]{Kr99}) that $\bH^{\infty}_0(T,\ell_2)$ is dense in $\bH^{\sigma}_2(T,\ell_2)$ for any $\sigma$.

Finally
we use $U^{\sigma}_2$ to denote the family of  $H^{\sigma}_2(\bR^d)$-valued $\cF_0$-measurable random variables $u_0$ having
$$
\|u_0\|_{U^{\sigma}_2}:=\left( \E  \|u_0\|^2_{H^{\sigma}_2}
\right)^{1/2}<\infty.
$$

\begin{lemma} \label{lemma 2}
Suppose $a>0$.
{\rm (i)} Let $h=(h^1,h^2,\cdots) \in L_2(\Omega\times [0,T],\cP, \ell_2)$. Then the equality
 $$
 I^{a}(\sum_{k=1}^{\infty} \int^{\cdot}_0 h^k (s) dW^k_s)(t)=\sum_{k=1}^{\infty} (I^{a} \int^{\cdot}_0 h^k (s) dW^k_s)(t)
 $$
 holds  for all $t\leq T$ (a.s.) and also in $L_2(\Omega\times [0,T])$.

 {\rm (ii)}  Suppose $
 h_n
 =(h^1_n,h^2_n,\cdots)$ converges to  $h=(h^1,h^2,\cdots)$  in $L_2(\Omega\times [0,T],\cP, \ell_2)$ as $n\to \infty$. Then, as $n\to \infty$,
 \begin{equation}
                      \label{eqn 8.30.2}
 \sum_{k=1}^\infty I^{a}\int^{\cdot}_0 h^k_n dW^k_s  \quad
  \hbox{converges to} \quad  \sum_{k=1}^\infty I^{a}\int^{\cdot}_0 h^k (s) dW^k_s
\end{equation}
in probability uniformly on $[0,T]$.
\end{lemma}

\begin{proof}
Since the series $\sum_{k=1}^{\infty} \int^t_0 h^k (s) dW^k_t$ converges in $L_2(\Omega \times [0,T])$, by (\ref{eqn 7.03.1}) we have
\begin{eqnarray*}
I^{a}\Big( \sum_{k=1}^{\infty} \int^{\cdot}_0 h^k (s) dW^k_s \Big)(t)&=&\lim_{n\to \infty} I^{a}\Big( \sum_{k=1}^{n} \int^{\cdot}_0 h^k (s) dW^k_s \Big)(t) \\
&=&\lim_{n\to \infty} \sum_{k=1}^{n} I^{a} \Big( \int^{\cdot}_0 h^k (s) dW^k_s \Big)(t) \quad \quad \text{in}\,\,\,L_2(\Omega \times [0,T]).
\end{eqnarray*}
Thus,  the series $\sum_{k=1}^{\infty} I^{a} \int^t_0 h^k (s) dW^k_s$ converges in $L_2(\Omega\times [0,T])$ and
\begin{equation}
          \label{eqn 7.04.1}
I^{a}\sum_{k=1}^{\infty} \int^t_0 h^k (s) dW^k_s=\sum_{k=1}^{\infty} I^{a} \int^t_0 h^k (s) dW^k_s
\end{equation}
in $L_2(\Omega\times [0,T])$, and thus the equality holds (a.e.).
Also by Burkholder-Davis-Gundy inequality
 \begin{eqnarray}
&&\bE \Big[ \sup_{t\leq T} \Big|\sum_{k=n}^{m} I^{a} \int^t_0 h^k (s) dW^k_s
\Big|^2 \Big] \nonumber\\
&\leq& N\bE \Big[ \sup_{t\leq T} \Big|\sum_{k=n}^{m}  \int^t_0 h^k (s) dW^k_s
\Big|^2 \Big] \leq N \sum_{k=n}^{m} \bE \Big[ \int^T_0    |h^k(s)|^2 ds \Big]
\to 0 \label{eqn 8.29.1}
\end{eqnarray}
as $n,m\to \infty$. Therefore the series $\sum_{k=1}^{\infty} I^{a} \int^t_0 h^k (s) dW^k_s$ converges in probability uniformly on $[0,T]$.
   It follows that both sides of  (\ref{eqn 7.04.1}) are continuous,  and therefore the equality above holds for all $t\leq T$ (a.s.).
\end{proof}

\begin{remark}
                           \label{remark 1}
Let $\sigma\in \bR$. Suppose $g_n \to g$ in $\bH^{\sigma}_2(T,\ell_2)$ as $n\to \infty$, and $\phi\in C^{\infty}_0 (\bR^d)$. Then $(g_n (t, \cdot),\phi)_{L_2} \to (g(t, \cdot),\phi)_{L_2}$ in $L_2(\Omega\times [0,T],\cP, \ell_2)$,
and therefore  Lemma \ref{lemma 2}(ii) holds  with $h_n(t):=(g_n(t, \cdot),\phi)_{L_2}$ and  $h(t):=(g(t, \cdot),\phi)_{L_2}$.
\end{remark}

\medskip

\begin{lemma}
               \label{derivative}
Let $\alpha>1/2$ and $g\in \bH^{\infty}_0(T,\ell_2)$. Then $I^{\alpha} \sum_{k=1}^\infty\int^{\cdot}_0 g^k (s) dW^k_s$ is differentiable in $t$ and  (a.s.) for all $t\leq T$
$$
\frac{\partial}{ \partial t}(I^{\alpha} \sum_{k=1}^\infty \int^{\cdot}_0 g^k (s) dW^k_s)(t)=\frac{1}{\Gamma(\alpha)}\sum_{k=1}^\infty \int^t_0 (t-s)^{\alpha-1}g^k(s)dW^k_s.
$$

\end{lemma}

\begin{proof}
We integrate the right hand side and then use the stochastic Fubini's theorem (see e.g. \cite{Kr11})  to get
\begin{eqnarray*}
\frac{1}{\Gamma(\alpha)} \sum_{k=1}^\infty\int^t_0\int^s_0 (s-r)^{\alpha-1}g^k(r)dW^k_r ds &=&\frac{1}{\Gamma(\alpha)}\sum_{k=1}^\infty \int^t_0\int^t_r (s-r)^{\alpha-1} ds g^k(r)dW^k_r\\
&=&\frac{1}{\alpha \Gamma(\alpha)} \sum_{k=1}^\infty\int^t_0 (t-r)^{\alpha}  g^k(r)dW^k_r.
\end{eqnarray*}
Similarly, by Lemma \ref{lemma 2}(i)
\begin{eqnarray*}
I^{\alpha} \sum_{k=1}^\infty \Big(\int^{\cdot}_0 g^k (s) dW^k_s \Big) (t)&=&\frac{1}{\Gamma(\alpha)}\sum_{k=1}^\infty\int^t_0 (t-s)^{\alpha-1}\left(\int^s_0 g^k(r)dW^k_r\right) ds\\
&=&\frac{1}{\alpha \Gamma(\alpha)} \sum_{k=1}^\infty\int^t_0 (t-r)^{\alpha}  g^k(r)dW^k_r.
\end{eqnarray*}
The lemma is proved.
\end{proof}
Lemma \ref{derivative} can be easily extended for any $g\in \bL_2(T,\ell_2)$.

For the remainder of this paper, we assume that \eqref{e:1.7} holds.
For $a\in \bR$, denote $a_+=\max\{a,0\}$.
Define
\begin{equation}
       \label{gamma0}
\gamma_0:=\frac{(2\gamma-1)_+}{\beta}<2.
\end{equation}
Note that since $\gamma<\beta+1/2$ we have
\begin{equation}
       \label{gamma1}
\gamma_0<2, \quad \text{and}\quad \gamma_0=0 \quad \text{if}\,\, \gamma\leq 1/2.
\end{equation}

\begin{definition}
                    \label{def 3.23}
We write $u\in \cH^{\sigma+2}_2(T)$ if $u\in \bH^{\sigma+2}_2(T)$, $u(0)\in U^{\sigma+1}_2$,  and for some $f\in \bH^{\sigma}_2(T)$ and $g\in \bH^{\sigma+\gamma_0}_2(T,\ell_2)$ it holds that
\begin{equation}
                   \label{eqn 7.08.1}
 \partial^{\beta}_tu (t,x)=f (t,x)
    +\sum_{k=1}^{\infty} \partial^{\gamma}_t \int^t_0 g^k (s,x)\, dW^k_s
 \end{equation}
 in the sense of distributions. That is, for any $\phi\in C^{\infty}_0 (\bR^d)$
 the equality
\begin{equation}
                  \label{sense of solution}
 (I^{1-\beta}(u-u(0))(t), \phi)_{L_2}=\int^t_0(f(s, \cdot), \phi)_{L_2}ds+\sum_{k=1}^{\infty} (I^{1-\gamma} \int^{\cdot}_0 (g^k(s, \cdot),\phi)_{L_2}\, dW^k_s)(t)
 \end{equation}
 holds for all $t\leq T$ (a.s.). In this case we write
 $$
 f=\bD u, \quad g=\bS u
 $$
 and define
\begin{equation}\label{e:3.10}
 \|u\|_{\cH^{\sigma+2}_2(T)}=\|u(0)\|_{U^{\sigma+1}_2}+\|u\|_{\bH^{\sigma+2}_2(T)}+\|\bD u\|_{\bH^{\sigma}_2(T)}+\|\bS u\|_{\bH^{\sigma+\gamma_0}_2(T,\ell_2)}.
 \end{equation}
 Finally define
 \begin{align}\label{e:ch0}
 \cH^{\sigma+2}_{2,0}(T)=\cH^{\sigma+2}_2(T) \cap \{u: u(0)=0\}.
 \end{align}
\end{definition}

\begin{remark}
By (\ref{e:1.10}), (\ref{2014.1.23.2}) and Lemma \ref{derivative}, (\ref{sense of solution}) is equivalent to
\begin{eqnarray*}
 (u(t, \cdot)-u(0, \cdot), \phi)_{L_2}&=&\frac{1}{\Gamma(\beta)}\int^t_0(t-s)^{\beta-1}(f(s, \cdot), \phi)_{L_2}ds\\
 &&+\frac{1}{\Gamma(1+\beta-\gamma)}\sum_{k=1}^{\infty} \int^t_0 (t-s)^{\beta-\gamma}  (g^k(s, \cdot),\phi)_{L_2}\, dW^k_s.
 \end{eqnarray*}
\end{remark}

\begin{lemma}
               \label{Banach}
{\rm (i)} $\cH^{\sigma+2}_2(T)$ is a Banach space.

{\rm (ii)} Let $u\in \cH^{\sigma+2}_2(T)$. Then $u$ is a continuous $H^{\sigma}_2$-valued process.

{\rm (iii)}  Assume that  $u\in \cH^{2}_2(T)$ and (\ref{eqn 7.08.1}) holds. Then (a.s.)
\begin{equation}
                      \label{eqn 8.27.1}
(k_{1-\beta}*\|u-u(0)-v\|^2_{L_2})(t)\leq 2\int^t_0 (f(s, \cdot),u(s, \cdot)-u(0, \cdot)-v(s, \cdot))_{L_2}ds \quad \hbox{for }  t\leq T ,
\end{equation}
where
\begin{eqnarray}\label{e:defv}
v(t,x)=\Gamma(1+\beta-\gamma)^{-1}\sum_{k=1}^\infty \int^t_0 (t-s)^{\beta-\gamma} g^k(s,x) dW^k_s.
\end{eqnarray}

\end{lemma}

\begin{proof}
(i) We only need to prove the completeness of the space. Let $u_n$, $n=1,2,\cdots$, be a Cauchy sequence in $\cH^{\sigma+2}_2(T)$ with
$$
f_n=\bD u_n, \quad g_n=\bS u_n.
$$
Then there exist $u, f, g, u_0$ so that $u_n, f_n, g_n, u_n(0)$ converge to $u,f,g,u_0$ respectively in their corresponding spaces. To prove $u_n \to u$ in $\cH^{\sigma+2}_2(T)$,  it suffices to  show  (\ref{sense of solution}) holds for all $t\leq T$ (a.s.).  Since the series
$\sum_{k=1}^\infty \int^t_0 (g^k_n (s),\phi) dW^k_s$ converges in probability uniformly on $[0,T]$, so does $(I^{1-\gamma} \sum_{k=1}^\infty \int^\cdot_0 (g^k_n(s),\phi) dW^k_s)(t)$. By Remark \ref{remark 1},  considering the limit of
$$
 (I^{1-\beta}(u_n-u_n(0)))(t), \phi)_{L_2}=\int^t_0(f_n(s, \cdot), \phi)_{L_2}ds+\sum_{k=1}^{\infty} (I^{1-\gamma} \int^{\cdot}_0 (g^k_n(s, \cdot),\phi)_{L_2}\, dW^k_s)(t)
$$
for $t\leq T$,
we get  (\ref{sense of solution})    for all $t\leq T$ (a.s.) since both sides of (\ref{sense of solution})  are continuous in $t$.

(ii) We only prove the case $\sigma=0$. The general case is covered by applying $(1-\Delta)^{\sigma/2}$ to (\ref{eqn 7.08.1}).  Denote $f=\bD u$ and $g=\bS u$. Notice that as an $L_2(\bR^d)$-valued process, $u(t)-u(0)$ satisfies
$$
k_{1-\beta}*(u(\cdot,x)-u(0,x))(t)=\int^t_0 f(s,x) ds+(k_{1-\gamma}* (\sum_{k=1}^\infty \int^{\cdot}_0 g^k(s,x)dW^k_s))(t) \quad \hbox{all } t\leq T \, (a.s.).
$$
Taking the convolution with $k_{\beta}$ and using
$$
\frac{\partial}{ \partial t}(k_{\beta}*\int^{\cdot}_0 f(s,x)ds)(t)=(k_{\beta}*f)(t,x),
$$
$$
\frac{\partial}{ \partial t}(k_{\beta}*k_{1-\gamma}*(\sum_{k=1}^\infty \int^{\cdot}_0 g^k (s,x)dw_s)) (t)=\sum_{k=1}^\infty \frac{1}{\Gamma(1+\beta-\gamma)}\int^{t}_0 (t-s)^{\beta-\gamma}g^k(s,x)dW^k_s,
$$
where the second equality is from Lemma \ref{derivative},
 we get for all $t\leq T$ (a.s.)
\begin{equation}
               \label{eqn 10}
u(t,x)-u(0,x)=(k_{\beta}*f)(t,x)+\sum_{k=1}^\infty \frac{1}{\Gamma(1+\beta-\gamma)}\int^{t}_0 (t-s)^{\beta-\gamma}g^k(s,x)dW^k_s.
\end{equation}
Hence the claim follows.

(iii)  Denote $w(t,x)=u(t,x)-u(0,x)-v(t,x)$. Then we have $\partial^{\beta}_tw(t,x)=f(t,x)$. Now we use the fact  (see e.g. \cite[Lemma 2.1]{Za}) that if $\kappa$ is a positive decreasing function on $[0,T]$
then
$$
\kappa*\|w\|^2_{L_2} (t)\leq 2\int^t_0 (\frac{\partial}{ \partial s}(\kappa*w)(s, \cdot),w(s, \cdot))_{L_2}ds.
$$
 We take (see \cite{Za}) a sequence of such functions $\kappa_n\in H^1_1([0,T])$ so that $\kappa_n \to k_{1-\beta}$ in $L_1([0,T]$ and $\frac{\partial}{ \partial t}(\kappa_n*w)\to \frac{\partial}{ \partial t}(k_{1-\beta}*w)$ in $L_2([0,T],L_2)$. Finally for (\ref{eqn 8.27.1}) it is enough to note that
$$
\kappa_n *\|w\|^2_{L_2} \to k_{1-\beta}*\|w\|^2_{L_2} \quad \text{in} \quad L_1([0,T])
$$
and both sides of (\ref{eqn 8.27.1}) are continuous in $t$. The lemma is proved.
\end{proof}

Recall that for any $\sigma$,
$$
\|u\|^2_{\bH^{\sigma}_2(t)}:=\bE\int^t_0\|u(s)\|^2_{H^{\sigma}_2}ds.
$$
\begin{proposition}
               \label{lemm 8.28.6}
Let $u\in \cH^{\sigma+2}_2(T)$.
Then for any $t\leq T$,
\begin{equation}
             \label{eqn 1.27.1}
(k_{1-\beta}*\bE\|u\|^2_{H^{\sigma}_2})(t)\leq N (\bE\|u(0)\|^2_{H^{\sigma}_2}+\|\bD u\|^2_{\bH^{\sigma}_2(t)}+\|\bS u\|^2_{\bH^{\sigma}_2(t,\ell_2)}+\|u\|^2_{\bH^{\sigma}_2(t)})
\leq N \|u\|^2_{\cH^{\sigma+2}_2(t)},
\end{equation}
where $N$ depends only on $T, \beta$ and $\gamma$. In particular, for any $t\leq T$,
\begin{equation}
                \label{eqn 1.27.2}
\|u\|^2_{\bH^{\sigma}_2(t)}\leq N \int^t_0 (t-s)^{-1+\beta}\|u\|^2_{\cH^{\sigma+2}_2(s)}\, ds.
\end{equation}
\end{proposition}

\begin{proof}
We only consider the case $\sigma=0$. In general,  one can consider $\Delta^{\sigma/2}u$ in place of $u$.
Denote $v$ as \eqref{e:defv} in Lemma \ref{Banach}. Then
by (\ref{eqn 8.27.1}),
$$
(k_{1-\beta}*\bE\|u\|^2_{L_2})(t)\leq 2 (k_{1-\beta}*\bE\|u(0)+v\|^2_{L_2})(t)+2\bE\int^t_0(f(s, \cdot),u(s, \cdot)-u(0, \cdot)-v(s, \cdot))_{L_2} ds.
$$
Note that by Fubini's theorem and Davis's inequality
\begin{eqnarray}
                              \nonumber
&&\bE\int^t_0\|v(s)\|^2_{L_2}ds = \int^{t}_0 \bE \|v(s)\|^2_{L_2}ds\\
&\leq& N \bE \int^t_0 \int^s_0 (s-r)^{2(\beta-\gamma)}
\|g(r, \cdot)\|^2_{L_2(\ell_2)} drds  \leq N\|g\|^2_{\bL_2(t,\ell_2)}.   \label{eqn 8.28.5}
\end{eqnarray}
For the last inequality we use the fact $2(\beta-\gamma)>-1$.
Therefore, by young's inequality
$$
\bE\int^t_0(f(s, \cdot),u(s, \cdot)-u(0, \cdot)-v(s, \cdot))_{L_2} ds\leq N(\bE\|u(0)\|^2_{L_2}+\|f\|^2_{\bL_2(t)}+\|g\|^2_{\bL_2(t,\ell_2)}+\|u\|^2_{\bL_2(t)}).
$$
Also,
$$
(k_{1-\beta}*\bE\|u(0)+v\|^2_{L_2})(t)\leq N \bE\|u(0)\|^2_{L_2}+ N k_{1-\beta}*\bE \|v\|^2_{L_2}(t),
$$
$$
k_{1-\beta}*\bE \|v\|^2_{L_2}(t)\leq N \int^t_0 (t-s)^{-\beta}\int^s_0 (s-r)^{2(\beta-\gamma)}\bE\|g(r, \cdot)\|^2_{L_2(\ell_2)} drds\leq N\|g\|^2_{\bL_2(t,\ell_2)}.
$$
This proves \eqref{eqn 1.27.1}.

Note that
the second equality in \eqref{eqn 1.27.1} just follows from the definition
\eqref{e:3.10} of $\|u\|^2_{\cH^{\sigma+2}_2(t)}$.
Hence to prove (\ref{eqn 1.27.2}) it is enough to consider a  convolution with $k_{\beta}$ and use \eqref{e:1.10}, which implies
  $$
  (k_{\beta}*k_{1-\beta}*\bE \|u\|^2_{L_2})(t)=\int^t_0 \bE \|u\|^2_{L_2}ds=\|u\|^2_{\bL_2(t)}.
  $$
  Hence the theorem is now proved.
\end{proof}

 Define
$$
P_{\beta,\gamma}(t,x)=\begin{cases}I^{\beta-\gamma}p (\cdot,x) (t)\quad &\text{if} \quad\beta\geq \gamma\\
\partial^{\gamma-\beta}_t p (\cdot,x) (t) \quad &\text{if}\quad \beta<\gamma.
\end{cases}
$$

\begin{lemma}
                  \label{lemma main estimate}
Let $g\in \bH^{\infty}_0(T,\ell_2)$ and $u$ be defined by
$$
u(t,x)=\sum_{k=1}^{\infty}\int^t_0 \int_{\bR^d}P_{\beta,\gamma}(t-s, x-y)g^k(s,y)dy dW^k_s.
$$
Let $\sigma \leq 2\wedge (\frac{1-2\gamma}{\beta} +2)$ if $\gamma \neq 1/2$, and $\sigma<2$ if $\gamma=1/2$. Then it holds that
$$
\bE \int^T_0 \|\Delta^{\sigma/2}u(t, \cdot)\|^2_{L_2} dt\leq N
 \|g\|^2_{\bL_2(T, \ell_2)}.
$$
In general, for any $\gamma_1 \in \bR$,
$$
\bE \int^T_0 \|u(t, \cdot)\|^2_{H^{\gamma_1+\sigma}_2} dt\leq N
 \|g\|^2_{\bH^{\gamma_1}_2(T, \ell_2)}.
$$
\end{lemma}

\begin{proof}
Let $a:=\gamma-\beta<1/2$. Recall that $E_{\alpha,\gamma}(t)$ is bounded on $(-\infty,0]$.
By the Fourier transform and
Lemma \ref{l:DaE}, we have  for any $\sigma\geq 0$
\begin{eqnarray*}
&&\bE\int^T_0\|\Delta^{\sigma/2}u(t, \cdot)\|^2_{L_2} dt\\ &\leq&N \sum_{k=1}^\infty \bE \int^T_0 \int^t_0\int_{\bR^d} |\xi|^{2\sigma}(t-s)^{-2a}E^2_{\beta, 1-a}(-|\xi|^2(t-s)^{\beta})|\hat{g}^k(s, \xi)|^2 d\xi ds\,dt\\
&\leq&N \|g\|^2_{\bL_2(T, \ell_2)}\\
&&+N \bE \int^T_0 \int^T_s\int_{|\xi|\geq 1} |\xi|^{2\sigma}(t-s)^{-2a}E^2_{\beta, 1-a}(-|\xi|^2(t-s)^{\beta})
|\hat{g}(s, \xi)|^2_{\ell_2} d\xi \,dt\,ds.
\end{eqnarray*}
By the substitution $r=|\xi|^{2/\beta} (t-s)$, the last term above is bounded by constant times of
$$
\bE \int_{|\xi|\geq 1}\int^T_0 |\hat{g}(s,\xi)|_{\ell_2}^2\int^{T|\xi|^{\frac{2}{\beta}}}_0 |\xi|^{2(\sigma+(2a-1)/{\beta})} r^{-2a}E^2_{\beta, 1-a}(-r^{\beta}) dr ds\,d\xi.
$$
Let $\gamma>1/2$.  Then,
since  $E_{\beta, 1-a}(-r^{\beta})$ is bounded on $[0,\infty)$ and $E_{\beta, 1-a}(-r^{\beta}) \le N r^{-\beta}$ if $r \geq 1$ (see (\ref{mittag})), we have
\begin{eqnarray*}
&&\int^{T|\xi|^{\frac{2}{\beta}}}_0 |\xi|^{2(\sigma+(2a-1)/{\beta})} r^{-2a}E^2_{\beta, 1-a}(-r^{\beta}) dr\\
&\leq& \int^{\infty}_0 r^{-2a} E^2_{\beta, 1-a}(-r^{\beta}) dr\\
&\le & N \left(\int_0^1 r^{-2a} dr+ \int_1^\infty r^{-2 \gamma} dr \right)<\infty.
\end{eqnarray*}
If $\gamma=1/2$, then  since $|\xi|\geq 1$ and $\sigma<2$,
\begin{eqnarray*}
&&\int^{T|\xi|^{\frac{2}{\beta}}}_0|\xi|^{2(\sigma+(2a-1)/{\beta})} r^{-2a}E^2_{\beta, 1-a}(-r^{\beta})  dr\\
&=&\int^{T|\xi|^{\frac{2}{\beta}}}_0|\xi|^{-2(2-\sigma)} r^{-2a}E^2_{\beta, 1-a}(-r^{\beta})  dr\\
&\leq& \int^1_0r^{-2a} E^2_{\beta, 1-a}(-r^{\beta})dr+N
|\xi|^{-2(2-\sigma)}  \int^{T|\xi|^{\frac{2}{\beta}}}_1 r^{-1} dr \\
 &\leq&N \int_0^1 r^{-2a} dr+N|\xi|^{-2(2-\sigma)} \ln|\xi|\leq N<\infty.
\end{eqnarray*}
The case $\gamma<1/2$ is treated similarly using $\sigma\leq 2$. Indeed,
\begin{eqnarray*}
&&\int^{T|\xi|^{\frac{2}{\beta}}}_0|\xi|^{2(\sigma+(2a-1)/{\beta})} r^{-2a}E^2_{\beta, 1-a}(-r^{\beta})  dr\\
&\leq& \int^1_0 r^{-2a}E^2_{\beta, 1-a}(-r^{\beta})dr+N
|\xi|^{2(\sigma+(2a-1)/{\beta})}  \int^{T|\xi|^{\frac{2}{\beta}}}_1 r^{-2\beta} dr\\
 &\leq&N\int_0^1 r^{-2a} dr+N |\xi|^{2\sigma-4}\leq N<\infty.
\end{eqnarray*}
Therefore the lemma is proved.
\end{proof}

Lemma \ref{lemma main estimate} says that $u$ (which is a solution of (\ref{time-space}) below) is smoother than $g$ by order $2\wedge ((1-2\gamma)\beta^{-1}+2)$ if $\gamma\neq 1/2$ and $2-\varepsilon$ if $\gamma=1/2$, where $\varepsilon>0$. Thus, for example, to estimate the second derivative of solution $u$ we need to assume
$$
\|g\|_{\bH^{\gamma_0}_2(T,\ell_2)}<\infty \quad \text{if}\,\, \gamma\neq 1/2,  \quad \text{and}\quad  \|g\|_{\bH^{\varepsilon}_2(T,\ell_2)}<\infty \quad \text{if}\,\, \gamma=1/2.
$$
Recall $\gamma_0=(2\gamma-1)_+/{\beta}<2$, which is defined in (\ref{gamma0}).

\vspace{3mm}

We first consider the equation
\begin{equation}
                    \label{time-space}
    \partial^{\beta}_tu (t,x)=\Delta u(t,x)+f(t,x)
    +\sum_{k=1}^{\infty} \partial^{\gamma}_t \int^t_0 g^k(s,x)\, dW^k_s.
    \end{equation}
Note that by letting $\beta\to 1$ and $\gamma \to 1$ we get the classical stochastic partial differentia equations.

\begin{lemma}
              \label{lemma formula}
 Let   $f\in \bL_2(T)$, $g\in \bH^{\infty}_0(T,\ell_2)$ and $u\in \bH^2_2(T)$. Then $u$ satisfies (\ref{time-space}) with  initial data $u_0\in U^1_2$ in the sense distributions (see Definition \ref{def 3.23}) if and only if
\begin{eqnarray}
u(t,x)&=&\int_{\bR^d}p(t,x-y)u_0(y)dy+\int^t_0\int_{\bR^d} q(t-s,x-y)f(s,y)dyds    \nonumber\\
&+&\sum_{k=1}^\infty \int^t_0 \int_{\bR^d}P_{\beta,\gamma}(t-s,x-y)g^k(s,y)dydW^k_s.     \label{right term}
\end{eqnarray}
\end{lemma}

\begin{proof}
Suppose $u$ satisfies (\ref{time-space}). Recall that for the solution of the (deterministic) equation
$$
\partial^{\beta}_t\bar{u}=\Delta \bar{u}+f, \quad u(0)=u_0
$$
is given by the formula
\begin{equation}
                         \label{eqn deterministic}
\bar{u}(t,x):=\int_{\bR^d}p(t,x-y)u_0(y)dy+ \int^t_0\int_{\bR^d} q(t-s,x-y)f(s,y)dyds.
\end{equation}
  In fact, in \cite[Section 5.2]{EIK} the representation (\ref{eqn deterministic}) is proved for sufficiently smooth $f$. This and the estimate of the solution obtained in \cite[Theorem 3.1]{Za} allow us to use an approximation (see the proof of Theorem \ref{thm laplace}) and get (\ref{eqn deterministic}) for general $f\in \bL_2(T)$.

Thus by considering $u-\bar{u}$, where $\bar{u}$ is defined above,
we may assume without loss of generality that $u_0=0$ and $f=0$.
Suppose first $\beta \leq \gamma$. Set  $a=\gamma-\beta$,
$$
v(t,x):=\sum_{k=1}^\infty\int^t_0 g^k(s,y)dW^k_s,
$$
and
$$
w(t,x):=D^{a}_tv(t,x)=\frac{1}{\Gamma(1-a)}\sum_{k=1}^\infty\int^t_0(t-s)^{-a}g^k(s,x)dW^k_s.
$$
Then $u-w$ satisfies the following fractional  diffusion equation
$$
\partial^{\beta}_t(u-w)=\Delta u=\Delta (u-w)+ \Delta w, \quad \quad (u-v)(0)=0.
$$
Thus by (\ref{eqn deterministic}) with $\Delta w$ in place of $f$, we have
$$
u(t,x)=
w(t, x)
+\int^t_0\int_{\bR^d}q(t-s,x-y)\Delta w(s,y)dyds.
$$
Nota that for any $s<t$
\begin{eqnarray*}
\int_{\bR^d}q(t-s,x-y)\Delta w(s,y)dy
&=&\int_{\bR^d} \Delta q(t-s,x-y) D_t^a v
(s,y) dy\\&=&\int_{\bR^d}\frac{\partial}{\partial t}p(t-s,x-y)D^a_tv(s,y)dy\\
&=&\frac{\partial}{\partial t}\int_{\bR^d}p(t-s,x-y)D^a_tv(s,y)dy.
\end{eqnarray*}
Indeed, the first equality is from Lemma \ref{lemma list}(iv) and the integration by parts, the second equality is from Lemma \ref{L:2.3} and the third equality is from Lemma \ref{lemma list}(v). Therefore $u(t,x)$ is equal to
\begin{eqnarray*}
&&D_t^a v(t,x)+ \int^t_0\frac{\partial}{\partial t}\int_{\bR^d}
p(t-s,x-y) D_t^a v (s,y) dyds\\
&=&\frac{\partial}{\partial t}\int^t_0\int_{\bR^d} p(t-s,x-y) D_t^a v
(s,y) dyds\\
&=&\frac{1}{\Gamma(1-a)}\sum_{k=1}^\infty\frac{\partial}{\partial t}\int^{t}_0\int_{\bR^d} p(t-s,x-y)
\int^s_0(s-r)^{-a}g^k(r,y)dW^k_r dyds\\
&=&\frac{1}{\Gamma(1-a)}\sum_{k=1}^\infty\frac{\partial}{\partial t}\int^{t}_0\int_{\bR^d} p(s,x-y)
\int^{t-s}_0(t-s-r)^{-a}g^k(r,y)dW^k_r dyds.
\end{eqnarray*}

Hence it is enough to prove
\begin{eqnarray}
          \nonumber
&&\frac{1}{\Gamma(1-a)}\sum_{k=1}^\infty \int^{t}_0\int_{\bR^d} p(s,x-y) \int^{t-s}_0(t-s-r)^{-a}g^k(r,y)dW^k_r
dyds\\
& =&\sum_{k=1}^\infty\int^t_0 \int^s_0 \int_{\bR^d} D_t^ap (s-r,x-y)g^k(r,y)dy dW^k_r ds.
\label{imp}
\end{eqnarray}
The latter equals
\begin{eqnarray}
&&\sum_{k=1}^\infty\int^t_0 \int_{\bR^d}\int^{t-r}_0  D_t^ap (s,x-y) ds g^k(r,y)dy dW^k_r \nonumber\\
&=&\sum_{k=1}^\infty\int^t_0 \int_{\bR^d}I^{1-a}p (t-r,x-y) g^k(r,y)dy dW^k_r \label{new 10}\\
&=&\frac{1}{\Gamma(1-a)}\sum_{k=1}^\infty\int^t_0\int_{\bR^d}\int^{t-r}_0(t-r-s)^{-a}p(s,x-y)ds g^k(r,y) dy dW^k_r.
\nonumber
\end{eqnarray}
For (\ref{new 10})  above we used the fact that  $\int^t_0 D_t^a p ds=I^{1-a}p$.
We thus get
\eqref{imp} using the stochastic Fubini's theorem (see \cite{Kr11}).

We now consider the case
$\beta\geq \gamma$. Put $a=\beta-\gamma$ and define
$$
v(t,x)=\int^t_0 g^k(s,x)dW^k_s, \quad w(t,x)=I^{a}v(t,x)=\frac{1}{a\Gamma(a)}\int^{t}_0
(t-s)^a g^k(s,x)dW^k_s.
$$
From this point on it is enough to repeat the case $\beta\leq \gamma$.  Indeed,
following the previous steps, we get
$$
u(t,x)=\frac{1}{a\Gamma(a)}\sum_{k=1}^\infty\frac{\partial}{\partial t}\int^{t}_0\int_{\bR^d} p(s,x-y)
\int^{t-s}_0(t-s-r)^{a}g^k(r,y)dW^k_r dyds.
$$
Note that
\begin{eqnarray*}
&&\sum_{k=1}^\infty\int^t_0 \int^s_0 \int_{\bR^d} I^ap (s-r,x-y)g^k(r,y)dy dW^k_r ds\\
&=&\sum_{k=1}^\infty\int^t_0 \int^{t-r}_0 \int_{\bR^d} I^ap (s,x-y)g^k(r,y)dy dsdW^k_r\\
&=&\frac{1}{a\Gamma(a)}\sum_{k=1}^\infty\int^t_0 \int^{t-r}_0 \int_{\bR^d} (t-s-r)^ap (r,x-y)g^k(r,y)dy
dsdW^k_r.
\end{eqnarray*}
This clearly proves the case $\gamma\geq \beta$.

 On the other hand, going backward of the above equalities one easily finds that if $u$ is given as in (\ref{right term}), then it satisfies (\ref{time-space}).  The proof of the lemma is now complete.
 \end{proof}

Recall $\gamma_0=(2\gamma-1)_+/{\beta}<2$. Fix $\varepsilon_0\in (0,1)$ and define
\begin{equation}\label{e:3.23}
\sigma_0:=\gamma_0+ (\varepsilon_0 1_{\gamma=1/2})=
\begin{cases}
 (2\gamma-1) /\beta \quad &\hbox{if } \gamma >1/2\\
 \varepsilon_0 &\hbox{if } \gamma =1/2\\
 0 &\hbox{if } \gamma <1/2.
 \end{cases}
 \end{equation}

\begin{theorem}
             \label{thm laplace}
   For any $f\in \bH^{\sigma}_2(T)$,  $g\in \bH^{\sigma+\sigma_0}_2(T,\ell_2)$ and $u_0\in U^{\sigma+1}_2$, equation (\ref{time-space}) has a unique solution $u\in \cH^{\sigma+2}_2(T)$, and for this solution we have
\begin{equation}
                    \label{main estimate1}
\|u\|_{\bH^{\sigma+2}_2(T)}\leq N\left(\|u_0\|_{U^{\sigma+1}_2}+\|f\|_{\bH^{\sigma}_2(T)}+\|g\|_{\bH^{\sigma+\sigma_0}_2(T,\ell_2)}\right),
\end{equation}
where $N$ depends only on $d$ and $T$.
\end{theorem}

\begin{proof}
Without loss of generality we only need to prove  the case $\sigma=0$.

For the deterministic equation, the theorem including the estimate  is proved in \cite{Za}. Thus the uniqueness for equation (\ref{time-space}) easily follows.

Define $w$ as in (\ref{eqn deterministic}). Then by considering $u-w$, we may assume without loss of generality that $u_0=0$ and $f=0$.

First, assume $g\in \bH^{\infty}_0(T,\ell_2)$. Then by Lemmas \ref{lemma main estimate} and \ref{lemma formula}, equation (\ref{time-space})
 has a unique solution $u\in \bH^2_2(T)$ and estimate \eqref{main estimate1} holds.

 For general $g\in  \bH^{\sigma_0}_2(T,\ell_2)$, take a sequence of $g_n\in \bH^{\infty}_0(T,\ell_2)$ so that $g_n \to g$ in  $\bH^{\sigma_0}_2(T,\ell_2)$. Define $u_n$ as the solution of equation (\ref{time-space}) with $g_n$ in place of $g$, that is,
 \begin{equation}
                            \label{approx}
  (I^{1-\beta}u_n)(t)=\int^t_0\Delta u_n (s)dt+\sum_{k=1}^{\infty} (I^{1-\gamma} \int^{\cdot}_0 g^k_n(s)\, dW^k_s)(t).                       \end{equation}
 Then by
 Lemmas \ref{lemma main estimate} and \ref{lemma formula}
 \begin{equation}
                   \label{eqn 7.03.2}
  \|u_n\|_{\bH^2_2(T)}\leq N \|g_n\|_{\bH^{\sigma_0}_2(T,\ell_2)},
  \end{equation}
  $$
  \|u_n-u_m\|_{\bH^2_2(T)}\leq N \|g_n-g_m\|_{\bH^{\sigma_0}_2(T,\ell_2)}.
  $$
  Thus $u_n\to u$ in $\bH^2_2(T)$ for some $u$.
  Letting $n\to \infty$ in (\ref{approx}) and using Remark \ref{remark 1}, we see that $u$ is a solution of  (\ref{time-space}).   Also we easily  get (\ref{main estimate1}) from  (\ref{eqn 7.03.2}). The
  theorem is proved.
\end{proof}

The following lemma is taken from \cite[Corllary 2]{YGD}.

\begin{lemma}(Gronwall's lemma)
                              \label{gronwall}
      Suppose $b>0$ and $a(t)$ is a nonnegative nondecreasing locally integrable  function on $[0,T)$, and  suppose $\eta(t)$ in nonnegative locally integrable on $[0,T)$ with
      $$
      \eta(t)\leq a(t)+b\int^t_0 (t-s)^{\beta-1}\eta(s)ds, \quad \forall t<T.
      $$
Then it holds that
$$
\eta(t)\leq a(t)E_{\beta}(b \Gamma(\beta)t^{\beta}).
$$
\end{lemma}

\section{\bf SPDE of divergence form type}

In this section,
we  study the equation of divergence type
 \begin{eqnarray}
                 \nonumber
  \partial^{\beta}_t u&=&D_i\left[
a^{ij}u_{x^j}+b^iu+f^i(u)\right]+cu+h(u)
\\
    &&+ \sum_{k=1}^\infty\partial^{\gamma}_t \int^t_0 (\sigma^{ijk}u_{x^ix^j}+\mu^{ik} u_{x^i}
    +\nu^ku +g^k(u))\, dW^k_s
    \label{eqn 7.04.5}
 \end{eqnarray}
where  the coefficients $a^{ij}, b^i, c, \sigma^{ijk}, \mu^{ik}, \nu^k$  are functions depending on
$(\omega,t,x)$ and the functions $f^i, h, g^k$ depend on $(\omega,t,x)$ and the unknown $u$.

For  a $\ell_2$-valued  continuous function $v$ in $\bR^d$, we define the space $C^{\alpha}(\ell_2)$, $\alpha\in [0,1]$ by the norm
$$
|v|_{C^{\alpha}(\ell_2)}=\sup_x |v|_{\ell_2}+ \sup_{x\neq y}\frac{|v(x)-v(y)|_{\ell_2}}{|x-y|^{\alpha}}.
$$
\vspace{3mm}
\begin{assumption}
               \label{ass 1}
(i) The coefficients $a^{ij}, b^i,  c, \sigma^{ijk}, \mu^{ik},
\nu^{ik}$ are
$ \cP \otimes \cB([0, T] \times \bR^d)$-measurable.

(ii) There exist constants $\delta, K_1>0$ so that for any $\xi\in
\bR^d$
$$
\delta |\xi|^2 \leq a^{ij}\xi_i\xi_j \leq K_1|\xi|^2 \quad \forall i,j,\omega,t,x.
$$
\begin{align}\label{e:aii}
|b^i|+|c|+|\sigma^{ij}|_{\ell_2}+|\mu^i|_{\ell_2}+|\nu|_{\ell_2}
\leq K_1  \quad \forall i,j,\omega,t,x.
\end{align}

(iii) $\sigma^{ijk}=0$ if $\gamma\geq 1/2$, and $\mu^{ik}=0$ if
$\gamma \geq 1/2+ \beta/2$ for every  $i,j, k, \omega,t,x$.
\end{assumption}

\medskip

Recall that $\sigma_0$ is the constant defined in \eqref{e:3.23}.

\medskip

\begin{assumption}
                      \label{ass 2}
(i) There exist constant $\kappa, K_2>0$,
$$
|\sigma^{ij}(t,\cdot)|_{C^1(\ell_2)}+|\mu^i(t,\cdot)|_{C^{|\sigma_0-1|+\kappa}(\ell_2)}+
|\nu(t,\cdot)|_{C^{|\sigma_0-1|+\kappa}(\ell_2)}\leq K_2,  \quad \quad \forall i,j,\omega,t.
$$

(ii) For any $\varepsilon>0$ there exists $K_3=K_3(\varepsilon, T)$ so that
\begin{align*}
&\|f^i(t, \cdot, u(\cdot))-f^i(t, \cdot, v(\cdot))\|_{L_2}+\|h(t, \cdot, u(\cdot))-h(t, \cdot, v(\cdot))\|_{H^{-1}_2}\\
&\qquad +\|g(t, \cdot, u(\cdot))-g(t, \cdot, v(\cdot))\|_{H^{-1+\sigma_0}_2(\ell_2)}\leq \varepsilon \|u\|_{H^{1}_2}+ K_3\|u\|_{L_2},
\end{align*}
for any $u,v \in H^{1}_2$ and $\omega, t$.
\end{assumption}

 See Example \ref{example 7} for  examples satisfying Assumption \ref{ass 2}(ii).

\medskip

 Denote
$$
f^i_0=f^i(t,x,0), \quad h_0=h(t,x,0), \quad g_0=g(t,x,0).
$$
We will use a well-known inequality (eg. \cite[Lemma 5.2]{Kr99})
\begin{align}\label{e:Kr9}
\|au\|_{H^{\sigma}_2}\leq N( \sigma ,d) |a|_{C^{\gamma}}\|u\|_{H^{\sigma}_2},
\end{align}
where $\gamma \geq |\sigma|$ is $\sigma$ is an integral, and otherwise $\gamma >|\sigma|$.

The following is the one of the two main results of this paper.
\begin{theorem}
                  \label{thm divtype}
Suppose Assumptions \ref{ass 1} and \ref{ass 2} hold. There exists $\kappa_0>0$ depending only on
$K,\gamma,\beta, d, T$ so that if $\sup_{\omega,i,j, t \le T}|\sigma^{ij}(t, \cdot)|_{C^1(\ell_2)}\leq \kappa_0$ then equation
(\ref{eqn 7.04.5}) with initial data $u_0\in U^0_2$ has a unique solution $u\in \bH^1_2(T)$, and for
this solution
\begin{equation}
         \label{estimate div}
\|u\|^2_{\cH^1_2(T)}\leq
N\left(\bE\|u_0\|^2_{L_2}+\|f^i_0\|^2_{\bL_2(T)}+\|h_0\|^2_{\bH^{-1}_2(T)}
+\|g_0\|^2_{\bH^{-1+\sigma_0}_2(\ell_2)}\right),
\end{equation}
where $N$ depends only on $\gamma,\beta, \delta, d, K$ and $T$.
\end{theorem}

\begin{proof}
 {\bf{A: Linear case}}. Let $f^i$, $h$ and
 $g^k$ depend only on  $(\omega, t, x)$.
Due to
the method of continuity and solvability result of Lemma \ref{thm
laplace},
it is enough to show that there exists
$\kappa_0>0$ so that if
$|\sigma^{ij} (t, \cdot)|_{C^1(\ell_2)}\leq \kappa_0$
and $u\in \cH^1_2(T)$ is a solution of \eqref{eqn 7.04.5},
then the estimate (\ref{estimate div}) holds.
 We refer the reader to the proof of \cite[Theorem 5.1]{Kr99} for details.

\medskip

\noindent
{\it Step 1.}  Assume $b^i=c=\mu^{ik}=\nu^k=0$.

By Theorem \ref{thm laplace}, the equation
$$
\partial^{\beta}_t v=\Delta v +D_if^i+h+
 \sum_{k=1}^\infty\partial^{\gamma}_t \int^t_0 (\sigma^{ijk}u_{x^ix^j}+g^k)
dW^k_s
$$
with initial data $u_0$ has a unique solution $v\in \cH^1_2(T)$, and
\begin{equation}
                    \label{01.24.11}
\|v\|^2_{\bH^1_2(T)}\leq
N\left(\bE\|u_0\|^2_{L_2}+\|D_if^i+h\|^2_{\bH^{-1}_2(T)}+\|\sigma^{ij}u_{x^ix^j}+g\|^2_{\bH^{-1+\sigma_0}_2(T,\ell_2)}\right).
\end{equation}

Note that for each $\omega$, $\bar{u}=u-v$ satisfies the
deterministic equation
$$
\partial^{\beta}_t \bar{u}=D_i(a^{ij}
u_{x^j})-\Delta
v=D_i(a^{ij}
\bar{u}_{x^j} +a^{ij}v_{x^j}-v_{x^i}) =
D_i(a^{ij}\bar{u}_{x^j}+\bar{f}^i),  \quad \bar{u}(0)=0
$$
where $\bar{f}^i=\sum_{j=1}^d a^{ij}v_{x^j}- v_{x^i}$.
By a result for the
deterministic equations (see \cite{Za}),
$$
\|\bar{u}\|_{\bH^1_2(T)}\leq N\|\bar{f}^i\|_{\bL_2(T)}\leq
N\|v\|_{\bH^1_2(T)}.
$$
Note that $\sigma^{ij}=0$ if $\gamma \geq 1/2$. If $\gamma<1/2$ then
$-1+\sigma_0=-1$ and so by \eqref{e:Kr9} for every $t \le T$
\begin{align*}
\|\sigma^{ij} (t, \cdot) u_{x^ix^j}(t, \cdot) \|_{H^{-1}_2(\ell_2)}&\leq
N|\sigma^{ij}(t, \cdot)|_{C^1(\ell_2)}\|u_{x^ix^j}(t, \cdot)\|_{H^{-1}_2}\\
&\leq
N\sup_{\omega,i,j, t \le T}|\sigma^{ij}(t, \cdot)|_{C^1(\ell_2)}\|u(t, \cdot)\|_{H^1_2}.
\end{align*}

 This and (\ref{01.24.11}) certainly lead to
\begin{eqnarray}
               \label{N_0}
\|u\|^2_{\bH^1_2(T)}&\leq& N_0
\sup_{\omega,i,j, t \le T}|\sigma^{ij}(t, \cdot)|_{C^1(\ell_2)}\|u\|^2_{\bH^1_2(T)}\\
&&+N_0\left(\|u_0\|^2_{U^0_2}+\|D_if^i+h\|^2_{\bH^{-1}_2(T)}
+\|g\|^2_{\bH^{-1+\sigma_0}_2(T,\ell_2)}\right), \nonumber
\end{eqnarray}
where $N_0$ depends only on
$\beta, \gamma, \delta, d, T$ and $K$. Note
that $\|D_if^i\|_{H^{-1}_2}\leq N\|f^i\|_{L_2}$.  Hence for the
desired estimate it is enough to take
$$\kappa_0=(2N_0)^{-1/2}.
$$

\medskip

\noindent
{\it Step 2.}  Take $\kappa_0$ from
{\it Step 1}, and assume
$\sup_{\omega,i,j, t \le T}|\sigma^{ij}(t, \cdot)|_{C^1(\ell_2)}\leq \kappa_0$. Then by the result of
{\it Step 1}, for each $t\leq T$,
\begin{eqnarray}
\|u\|^2_{\bH^1_2(t)} &\leq&
N\|b^iu+f^i\|^2_{\bL_2(t)}+N\|cu+h\|^2_{\bH^{-1}_2(t)}
 + N\|\mu^iu_{x^i}+\nu
u+g\|^2_{\bH^{-1+\sigma_0}_2(t,\ell_2)}.       \label{01.26.1}
\end{eqnarray}
Note that, by \eqref{e:aii}
$$
\|b^iu\|_{L_2}+\|cu\|_{H^{-1}_2} \leq \|b^iu\|_{L_2}+\|cu\|_{L_2}\leq N\|u\|_{L_2}\leq \varepsilon \|u\|_{H^1_2}+N\|u\|_{H^{-1}_2}.
$$
Also, since $-1+\sigma_0<1$, by a Sobolev embedding theorem and \eqref{e:Kr9}, for any $\varepsilon>0$
$$\|\nu u\|_{H^{\sigma_0-1}_2}\leq N|\nu|_{C^{|\sigma_0-1|+\kappa}(\ell_2)}\|u\|_{H^{\sigma_0-1}_2}\leq N \varepsilon \|u\|_{H^1_2}+N(\varepsilon)\|u\|_{H^{-1}_2}.
$$
By the assumption, $\nu^i=0$ unless $-1+\sigma_0<0$. Therefore,
\begin{eqnarray*}
 \|\mu^i u_{x^i}\|_{H^{-1+\sigma_0}_2}\leq N |\mu^i|_{C^{|\sigma_0-1|+\kappa}(\ell_2)}\|u_x\|_{H^{-1+\sigma_0}_2}
 \leq N\|u\|_{H^{\sigma_0}_2}\leq N \varepsilon \|u\|_{H^1_2}+N(\varepsilon)\|u\|_{H^{-1}_2}.
\end{eqnarray*}
Taking sufficiently small $\varepsilon > 0$ and using (\ref{01.26.1}), we get
for any $t\leq T$,
\begin{equation}
                   \label{e:4.8}
\|u\|^2_{\bH^1_2(t)}\leq N \|u\|^2_{\bH^{-1}_2(t)}+ N (\|u_0\|^2_{U^0_2}+\|f^i\|^2_{\bL_2(t)}+\|h\|^2_{\bH^{-1}_2(t)}+\|g\|^2_{\bH^{-1+\sigma_0}_2(t,\ell_2)}).
\end{equation}
Since
$$
\bD u = D_i (a^{ij}u_{x_j}+b^i u +f)+cu+h, \quad \bS u = \sigma^{ij}u_{x_ix_j} + \mu^{i}u_{x_i}+\nu u +g
$$
and  $D_i: H^{\gamma}_2 \to H^{\gamma-1}_2$ is a bounded operator,
we have
\begin{equation}\label{e:4.9}
 \| \bD u \|^2_{\bH^{-1}_2(t)} \leq N \Big( \| u\|^2_{\bH^1_2 (t)}
 + \| f\|^2_{\bL_2(t)} \Big)
  \quad \hbox{and} \quad
  \| \bS u \|^2_{\bH^{-1}_2(t, \ell_2)} \leq N \Big( \| u\|^2_{\bH^1_2 (t)}
   + \| g\|^2_{\bH^{-1}_2(t, \ell_2)} \Big) .
\end{equation}
This,  \eqref{e:4.8}   and
Proposition  \ref{lemm 8.28.6} (with $\sigma=-1$) yield
$$
(k_{1-\beta}*\bE\|u\|^2_{H^{-1}_2})(t)\leq N (\|u_0\|^2_{U^0_2}+\|f^i\|^2_{\bL_2(t)}+\|h\|^2_{\bH^{-1}_2(t)}+\|g\|^2_{\bH^{-1+\sigma_0}_2(t,\ell_2)}+\|u\|^2_{\bH^{-1}_2(t)}).
$$
Denote $\eta(t)=\bE\int^t_0\|u(s, \cdot)\|^2_{H^{-1}_2}ds$. Then taking the convolution with $k_{\beta}$ (recall $k_{\beta}*k_{1-\beta}=1$),  we get
\begin{equation}
                     \label{gronwall18}
\eta(t)\leq N M+ N(k_{\beta}*\eta)(t),
\end{equation}
where
$$M:=\|u_0\|^2_{U^{0}_2}+\|f^i\|^2_{\bL_2(T)}+\|h\|^2_{\bH^{-1}_2(T)}+\|g\|^2_{\bH^{-1+\sigma_0}_2(T,\ell_2)}.
 $$
 Consequently (\ref{e:4.8}), (\ref{gronwall18}) and Gronwall's lemma (Lemma \ref{gronwall})  finish the proof for the linear case.

 \vspace{2mm}

 {\bf{B: Non-linear case}}. The proof is  identical to that of the non-divergence type case. See the proof of Theorem \ref{main} below (it is enough to replace $\sigma$ by $-1$).
\end{proof}

\section{\bf SPDE of non-divergence form type}

\subsection{$L_2$-theory for fractional time SPDE of non-divergence form type}

In this subsection,
we  study the equation of non-divergence type
 \begin{eqnarray}
                 \nonumber
  \partial^{\beta}_t u&=&\left(a^{ij}u_{x^ix^i}+b^iu+cu+f(u)\right)
\\
    &&+ \sum_{k=1}^\infty\partial^{\gamma}_t \int^t_0 (\sigma^{ijk}u_{x^ix^j}+\mu^{ik} u_{x^i}
    +\nu^ku +g^k(u))\, dW^k_s
    \label{non-div eqn}
 \end{eqnarray}
where  the coefficients $a^{ij}, b^i, c, \sigma^{ijk}, \mu^{ik}, \nu^k$  are functions depending on
$(\omega,t,x)$ and the functions $f, g^k$ depend on $(\omega,t,x)$ and the unknown $u$.
Recall that $\sigma_0$ is the constant defined by \eqref{e:3.23}

\begin{assumption}
                \label{ass 3}

(i) The coefficients $a^{ij}$ are uniformly continuous in $x$, that is for any $\varepsilon>0$, there exists $\delta>0$ so that
$$
|a^{ij}(t,x)-a^{ij}(t,y)| < \varepsilon, \quad \forall i,j,\omega,t
$$
whenever $|x-y|<\delta$.

(ii) H\"older continuity of $a^{ij}$ when $\sigma \neq 0$ : if $\sigma \neq 0$, there exists  constants $\kappa, K_1>0$ so that
\begin{equation}
                    \label{weaken}
|a^{ij}(t,\cdot)|_{C^{|\sigma|+\kappa}}<K_1, \quad  \forall i,j,\omega,t.
\end{equation}

(iii) For any $i,j,\omega$ and  $t$
\begin{equation}
                      \label{weaken2}
|b^i(t,\cdot)|_{C^{|\sigma|+\kappa}}+| c(t,\cdot)|_{C^{|\sigma|+\kappa}} +|\sigma^{ij}(t,\cdot), \mu^i(t,\cdot),\nu(t,\cdot)|_{C^{|\sigma+\sigma_0|+\kappa}(\ell_2)} \leq K_2 < \infty.
\end{equation}

(iv)  $|\sigma^{ij}(t,x)|_{\ell_2}\leq \kappa_0$,
where $\kappa_0$ is
the constant in Theorem \ref{thm divtype}.

(v) For any $\varepsilon>0$ there exists $K_3=K_3(\varepsilon)$ so that
\begin{equation}
                        \label{Sobolev}
\|f(t, \cdot, u(\cdot))-f(t, \cdot, v(\cdot))\|_{H^{\sigma}_2}+\|g(t, \cdot, u(\cdot))-g(t, \cdot, v(\cdot))\|_{H^{\sigma+\sigma_0}_2}\leq \varepsilon \|u\|_{H^{\sigma+2}_2}+ K_3\|u\|_{H^{\sigma+1}_2},
\end{equation}
for any $u,v \in H^{\sigma+2}_2$.
\end{assumption}

\begin{remark}
If $\sigma$ is  integer then one can slightly weaken (\ref{weaken}) and (\ref{weaken2}) and take $\kappa=0$ as is done in \cite{Kr99}.
\end{remark}

\begin{example}
               \label{example 7}
  (i) Let $\delta:=\sigma+2-d/2>0$ and $f_0=f_0(x)\in H^{\sigma}_2$. Take
  $$
  f(x,u)=f_0(x) \sup_x |u|.
  $$
  Then by a Sobolev embedding
  \begin{eqnarray}
                 \nonumber
  \|f(u)-f(v)\|_{H^{\sigma}_2} &\leq& \|f_0\|_{H^{\sigma}_2} \sup_x |u-v|\leq N\|u-v\|_{H^{\sigma+2-\delta/2}_2}\\
  &\leq&\varepsilon  \|u-v\|_{H^{\sigma+2}_2}+K \|u-v\|_{H^{\sigma}_2}. \label{same}
  \end{eqnarray}

  (ii) Fix $\varepsilon>0$ and  take $\delta\in (0,1)$ and a random $C^{|\sigma|+\varepsilon}$-function $a(t,x)$. Let
  $$
  f(t,x,u)=a(t,x)(-\Delta)^{\delta}u, \quad  |a(t,\cdot)|_{C^{|\sigma|+\varepsilon}}\leq K, \quad \forall \omega,t.
  $$
  Then the  argument used to prove (\ref{same}) easily leads to (\ref{Sobolev}).
\end{example}

\vspace{3mm}
Denote
$$
f_0=f(t,x,0), \quad g_0=g(t,x,0).
$$

Here is the second main result of this paper.
\begin{theorem}
                   \label{main}
 Let $\sigma \in \bR$ and Assumptions \ref{ass 1} and  \ref{ass 3} hold. Then for any $f\in \bH^{\sigma}_2(T)$, $g\in \bH^{\sigma+\sigma_0}_2(T,\ell_2)$  and  $u_0\in U^{\sigma+1}_2$,  equation (\ref{non-div eqn}) admits a unique solution $u\in \cH^{\sigma+2}_2(T)$, and for this solution we have
\begin{equation}
                    \label{main estimate}
\|u\|_{\bH^{\sigma+2}_2(T)}\leq N\left(\|u_0\|_{U^{\sigma+1}_2}+\|f_0\|_{\bH^{\sigma}_2(T)}+\|g_0\|_{\bH^{\sigma+\sigma_0}_2(T,\ell_2)}\right),
\end{equation}
where $N$ depends only on $d, \beta, \gamma, \delta, K$ and $T$.
\end{theorem}

\begin{proof}
By considering $u-v$ if needed, where $v$ is the solution of
$ \partial^{\beta}_t v= \Delta v$ with $v(0)=u_0$,
we may assume without loss of generality that $u_0=0$.

\medskip
{\bf{A: Linear case}}.  Let $f^i$, $h$ and
 $g^k$ depend only on   $(\omega, t, x)$.
Due to the method of continuity we only need to prove that the estimate
\eqref{main estimate} holds given that a solution already exists.

\medskip

\noindent
{\it Step 1.}
  Assume that  all the coefficients  are independent of $x$, so that equation (\ref{non-div eqn}) is of type (\ref{eqn 7.04.5}). By
applying the operator $(1-\Delta)^{(\sigma+1)/2}$ to equation (\ref{non-div eqn}), one can simplify the problem to the case  $\sigma=-1$. In this case
all the claims follow from Theorem \ref{estimate div}.

\medskip

\noindent
{\it Step 2.}
Next, we weaken the condition in
{\it Step 1} by proving that there exists a $\varepsilon_1\in (0,\kappa_0]$ so that the theorem holds if
\begin{equation}
        \label{purbation}
|a^{ij}(t,x)-a^{ij}(t,y)|+|\sigma^{ij}(t,x)-\sigma^{ij}(t,y)|_{\ell_2}\leq \varepsilon_1  \quad \forall i,j,\omega,t,x,y.
\end{equation}

Fix $x_0\in \bR^d$ and denote
$$
a^{ij}_0(t,x)=a^{ij}(t,x_0), \quad \sigma^{ij}_0(t,x)=\sigma^{ij}(t,x_0).
$$
Note that equation (\ref{non-div eqn}) can be written as
$$
\partial^{\beta}_t u=\left(a^{ij}_0u_{x^ix^i}+\bar{f}\right)
    + \sum_{k=1}^\infty\partial^{\gamma}_t \int^t_0 (\sigma^{ijk}_0u_{x^ix^j}+\bar{g}^k)\, dW^k_s,
$$
where
$$
\bar{f}:=(a^{ij}-a^{ij}_0)u_{x^ix^j}+b^iu_{x^i}+cu+f,
$$
$$
\bar{g}^k:=(\sigma^{ijk}-\sigma^{ijk}_0)u_{x^ix^j}+\mu^{ik}u_{x^i}+\nu^k u +g^k.
$$
Note that the coefficients $a^{ij}_0$ and $\sigma^{ij}_0$  are independent of $x$.  By the result of
{\it Step 1}, for each $t\leq T$,
\begin{equation}
                \label{middle estimate}
\|u\|_{\bH^{\sigma+2}_2(t)}\leq N\left(\|u_0\|_{U^{\sigma+1}_2}+\|\bar{f}\|_{\bH^{\sigma}_2(t)}+\|\bar{g}\|_{\bH^{\sigma+\sigma_0}_2(t,\ell_2)}\right).
\end{equation}
To estimate $\bar{f}$ and $\bar{g}$ we use the following well known embedding result: for
$0\leq \alpha_1 \leq \alpha_2$ and $\alpha_2>0$
\begin{equation}
                   \label{holder}
|v|_{C^{\alpha_1}}
\le N |v|^{1-\alpha_3}_{C^0}|v|^{\alpha_3}_{C^{\alpha_2}}, \quad \alpha_3:=\frac{\alpha_1}{\alpha_2}.
\end{equation}
If $\sigma=0$, then
$$
\|(a^{ij}(t, \cdot)-a^{ij}_0(t))u_{x^ix^j}(t, \cdot)\|_{L_2}\leq N \sup_x|a^{ij}(t, \cdot)-a^{ij}_0(t)| \cdot \|u(t, \cdot)\|_{H^2_2},
$$
and otherwise,
by first using \eqref{e:Kr9}
 and then taking $\alpha_1=|\sigma|+\kappa/2$ and $\alpha_2=|\sigma|+\kappa$ in (\ref{holder}),
\begin{align*}
\|(a^{ij}(t, \cdot)-a^{ij}_0(t))u_{x^ix^j}(t, \cdot)\|_{H^{\sigma}_2}
&\leq N |a^{ij}(t, \cdot)-a^{ij}_0 (t)|_{C^{|\sigma|+\kappa/2}}  \cdot \|u(t, \cdot)\|_{H^{\sigma+2}_2}\\
&\leq N \sup_x|a^{ij}(t, \cdot)-a^{ij}_0 (t)|^{\delta} \cdot \|u(t, \cdot)\|_{H^{\sigma+2}_2}.
\end{align*}
where
$$
 \delta:=1-\frac{|\sigma|+\kappa/2}{|\sigma|+\kappa}>0.
$$

The term $\|(\sigma^{ij}(t, \cdot)-\sigma^{ij}_0(t))u_{x^ix^j}(t, \cdot)\|_{H^{\sigma+\sigma_0}_2}$ and others can be handled similarly. For instance,  by \eqref{e:Kr9}
$$
\|b^i(t, \cdot)u_{x^i}(t, \cdot)\|_{H^{\sigma}_2}
\le N |b^i(t, \cdot)|_{C^{|\sigma|+\kappa}}\cdot \|u(t, \cdot)\|_{H^{\sigma+1}_2}\leq \varepsilon \|u(t, \cdot)\|_{H^{\sigma+2}_2}+N\|u(t, \cdot)\|_{H^{\sigma}_2},
$$
and, since $\mu^i=0$ unless $\sigma_0 <1$,
\begin{align*}
\|\mu^i (t, \cdot) u_{x^i}(t, \cdot)\|_{H^{\sigma+\sigma_0}_2(\ell_2)}&\leq N |\mu(t, \cdot)|_{C^{|\sigma+\sigma_0|+\kappa}(\ell_2)} \|u(t, \cdot) \|_{H^{\sigma+\sigma_0+1}_2}
\\&\leq \varepsilon \|u(t, \cdot)\|_{H^{\sigma+2}_2}+N\|u(t, \cdot)\|_{H^{\sigma}_2}.
\end{align*}
Hence, from (\ref{middle estimate}) it follows
\begin{eqnarray*}
\|u\|_{\bH^{\sigma+2}_2(t)} &\leq& N\left(\sup_{x,t \le T} |a^{ij}-a^{ij}_0|^{\delta}+ \sup_{x,t \le T} |\sigma^{ij}-\sigma^{ij}_0|^{\delta}_{\ell_2} +\varepsilon\right)\cdot \|u\|_{\bH^{\sigma+2}_2(t)} \\
&&+N\|u\|_{\bH^{\sigma}_2(t)}+N\left(\|u_0\|_{U^{\sigma+1}_2}+\|f\|_{\bH^{\sigma}_2(t)}+\|g\|_{\bH^{\sigma+\sigma_0}_2(t,\ell_2)}\right).
\end{eqnarray*}
Take $\varepsilon, \varepsilon_1>0$ so that $\varepsilon, \varepsilon_1< (4N)^{-1}$. If we assume (\ref{purbation}) then for each $t\leq T$,
\begin{equation}
                \label{eqn 1.26.5}
\|u\|_{\bH^{\sigma+2}_2(t)}\leq N \|u\|_{\bH^{\sigma}_2(t)}+\left(\|u_0\|_{U^{\sigma+1}_2}+\|f\|_{\bH^{\sigma}_2(t)}+\|g\|_{\bH^{\sigma+\sigma_0}_2(t,\ell_2)}\right).
\end{equation}
Just as \eqref{e:4.9} and the rest of the argument
in the proof of Theorem \ref{thm divtype},
this and Gronwall's lemma (Lemma \ref{gronwall}) lead to the desired estimate.

\medskip

\noindent
{\it Step 3.}
 General linear case without condition \eqref{purbation}. Extension of
{\it Step 2} to the general case is quite straightforward
 and can be found for example in the proof of  \cite[Theorem 5.1]{Kr99}. One introduces a partition of unity $\{\zeta_n :n=1,2,\cdot\}$ of $C^{\infty}_0(\bR^d)$-functions so that (\ref{purbation}) holds on each support of $\zeta_n$. Then one estimates $u\zeta_n$ using the result of {\it Step 2} and by summing up these estimate one easily gets (\ref{eqn 1.26.5}), which is sufficient for our estimate.

\vspace{2mm}

{\bf{B: Non-linear case}}.  We modify the proof of \cite[Theorem 5.1]{Kr99}.
Recall that $\cH^{\sigma+2}_{2,0}(T)$ is defined in \eqref{e:ch0}.
  For each $u\in \cH^{\sigma+2}_{2,0}(T)$  consider the equation
\begin{eqnarray}
                 \nonumber
  \partial^{\beta}_t v&=&\left(a^{ij}
  v_{x^ix^i}+b^iv+cv+f(u)\right)
\\
    &&+ \sum_{k=1}^\infty\partial^{\gamma}_t \int^t_0 (\sigma^{ijk}v_{x^ix^j}+\mu^{ik} v_{x^i}
    +\nu^kv +g^k(u))\, dW^k_s
    \nonumber
    \end{eqnarray}
with initial data $v(0)=0$. By the above results, this equation has a unique solution
$v\in \cH^{\sigma+2}_{2, 0}(T)$.
By denoting $v=\cR u$  we can define an operator
$\cR : \,\,\cH^{\sigma+2}_{2,0}(T) \to \cH^{\sigma+2}_{2,0}(T)$.

Note that due to the interpolation  $\|\xi\|_{H^{\sigma+1}_2}\leq \varepsilon \|\xi\|_{H^{\sigma+2}_2}+N\|\xi\|_{H^{\sigma}_2}$, (\ref{Sobolev}) is equivalent to
\begin{equation}
          \label{Sobolev2}
\|f(t, \cdot, u(\cdot))-f(t, \cdot, v(\cdot))\|_{H^{\sigma}_2}+\|g(t, \cdot, u(\cdot))-g(t, \cdot, v(\cdot))\|_{H^{\sigma+\sigma_0}_2}\leq \varepsilon \|u\|_{H^{\sigma+2}_2}+ K\|u\|_{H^{\sigma}_2}
\end{equation}
for some $K=K(\eps)>0$.

By the results for the linear case and (\ref{Sobolev2}) and Proposition \ref{lemm 8.28.6}, for each $t\leq T$,
\begin{eqnarray*}
\|\cR u-\cR v\|^2_{\cH^{\sigma+2}_2(t)}&\leq& N \|f(u)-f(v)\|^2_{\bH^{\sigma}_2(t)}+N\|g(u)-g(v)\|^2_{\bH^{\sigma+\sigma_0}_2(t,\ell_2)}\\
&\leq& N\varepsilon^2\|u-v\|^2_{\cH^{\sigma+2}_2(t)}+NK^2\|u-v\|^2_{\bH^{\sigma}_2(t)}\\
&\leq& N_0\varepsilon^2\|u-v\|^2_{\cH^{\sigma+2}_2(t)}+N_1\int^t_0 (t-s)^{-1+\beta}\|u-v\|^2_{\cH^{\sigma+2}_2(s)}\,ds,
\end{eqnarray*}
where $N_1$ depends also on $\varepsilon$.
Next, we fix $\varepsilon$ so that  $\theta:=N_0\varepsilon^2<1/4$. Then repeating the above inequality and using the identity
\begin{eqnarray*}
\int^t_0 (t-s_1)^{-1+\beta}\int^{s_1}_0 (s_1-s_2)^{-1+\beta}\cdots \int^{s_{n-1}}_0 (s_{n-1}-s_n)^{-1+\beta} ds_n \cdots ds_1
=\frac{\Gamma(\beta)}{\beta\Gamma(n\beta+1)} t^{n\beta},
\end{eqnarray*}
we get
\begin{eqnarray*}
&&\|\cR^m u-\cR^m v\|^2_{\cH^{\sigma+2}_2(t)}
\\
&&\leq \sum_{k=0}^m \begin{pmatrix} m\\ k\end{pmatrix}
\theta^{m-k} (T^{\beta}N_1)^k \frac{\Gamma(\beta)}{\beta\Gamma(k\beta+1)}  \, \|u-v\|^2_{\cH^{\sigma+2}_2(t)}\\
&&\leq 2^m \theta^m \left[\max_k \left( (\theta^{-1}T^{\beta}N_1)^k\frac{\Gamma(\beta)}{\beta\Gamma(k\beta+1)}\right)\right]\, \|u-v\|^2_{\cH^{\sigma+2}_2(t)}\\
&&\leq
\frac{1}{2^m} N_2 \|u-v\|^2_{\cH^{\sigma+2}_2(t)}.
\end{eqnarray*}
For the second inequality above we use $\sum_{k=0}^m \begin{pmatrix} m\\ k\end{pmatrix}=2^m$. It follows that if $m$ is sufficiently large then
$\cR^m$ is a contraction in $\cH^{\sigma+2}_{2,0}(T)$, and this yields all the claims. The theorem is proved.
\end{proof}

\subsection{An application to SPDE driven by space-time white noise}

In this subsection, we consider a SPDE driven by space-time white noise.
We consider
\begin{equation}
     \label{space-time}
  \partial^{\beta}_t u=\left(a^{ij}u_{x^ix^j}+b^iu_{x^i}+cu+f(u)\right)
    + \sum_{k=1}^\infty\partial^{\gamma}_t \int^t_0  h(u) \, dB_t
    \end{equation}
    where  the coefficients $a^{ij}, b^ic$ and   are functions depending on
$(\omega,t,x)$,  the functions $f$ and $h$ depends on $(\omega,t,x)$ and the unknown $u$, and
    $B_t$ is a space-time white noise.

    Let $\{\eta^k:k=1,2,\cdots\}$ be an orthogonal basis of $L_2(\bR^d)$. Then (at least formally)
    $$
    B_t=\sum_{k=1}^{\infty} \eta^k W^k_t
    $$
    where
    $W^k_t:=(B_t,\eta^k)_{L_2}$are independent one dimensional Wiener processes. Hence one can rewrite (\ref{space-time}) as
    $$
    \partial^{\beta}_t u= \left(a^{ij}u_{x^ix^j}+b^iu_{x^i}+cu+f(u)\right)
    + \sum_{k=1}^\infty\partial^{\gamma}_t \int^t_0  h(u) \eta^k \, dW^k_t.
    $$

Denote
$$
g^k(t,x,u)=h(t,x,u)\eta^k(x).
$$
To apply Theorem \ref{main}, we only need to find  $\sigma$ and conditions on $h$ so that (\ref{Sobolev}) holds.
The following lemma is a consequence of  \cite[Lemma 8.4]{Kr99}.
\begin{lemma}
          \label{lemma 1.26}
Let $\gamma<-1/2$. Then
$$
\|g(t, \cdot, u(\cdot))-g(t, \cdot, v(\cdot))\|_{H^{\gamma}_2(\ell_2)}\leq N \|h(t, \cdot, u(\cdot))-h(t, \cdot, v(\cdot))\|_{L_2}.
$$
\end{lemma}

The following is an easy consequence of Theorem \ref{main} and Lemma \ref{lemma 1.26}.
Recall that $f_0=f(t,x,0)$. We also denote
$
h_0=h(t,x,0).
$
\begin{corollary}
              \label{thm space-time}
  Let
  \begin{equation}
                   \label{eqn last}
  \sigma+\sigma_0<-1/2, \quad \sigma+2>0.
  \end{equation}
  Assume
  $$
  |f(t,x,v_1)-f(t,x,v_2)|+|h(t,x,v_1)-h(t,x,v_2)|\leq K|v_1-v_2| \quad \forall \omega,t, x, v_1, v_2,
  $$
  and  Assumptions \ref{ass 1} and  \ref{ass 3} hold  with $\sigma$ satisfying (\ref{eqn last}). Then
  equation (\ref{space-time}) with initial data $u_0\in U^{\sigma+1}_2$ has a unique solution $u$, and for this solution we have
  $$
  \|u\|_{\bH^{\sigma+2}_2(T)}\leq N
  \left(\|f_0\|_{\bH^{\sigma}_2(T)}+\|h_0\|_{\bL_2(T)}+\|u_0\|_{U^{\sigma+1}_2}\right).
  $$
  \end{corollary}
\begin{proof}
It is enough to note that, since $\sigma+2>0$,
\begin{eqnarray*}
&&\|f(t, \cdot, u(\cdot))-f(t, \cdot, v(\cdot))\|_{H^{\sigma}_2}+\|g(t, \cdot, u(\cdot))-g(t, \cdot, v(\cdot))\|_{H^{\sigma+\sigma_0}_2(\ell_2)}\\
&\leq& N\|u-v\|_{L_2}\leq
 \varepsilon \|u-v\|_{H^{\sigma+2}_2}+K\|u-v\|_{H^{\sigma}_2}.
 \end{eqnarray*}
 The corollary is proved.
\end{proof}

The constant $\sigma+2$ gives
the regularity of solution $u$.
To see how smooth the above solution is,
we recall
$$
\sigma_0:=\gamma_0+ (\varepsilon_0 1_{\gamma=1/2})
=\begin{cases}
 (2\gamma-1) /\beta \quad &\hbox{if } \gamma >1/2\\
 \varepsilon_0 &\hbox{if } \gamma =1/2\\
 0 &\hbox{if } \gamma <1/2.
 \end{cases}
$$
Since $\sigma+2=(\sigma+\sigma_0)+(2-\sigma_0)<-1/2+(2-\sigma_0)$, it follows
$$
\sigma+2 < \begin{cases}
 \frac{3}{2}-\frac{2\gamma-1}{\beta} \quad &\hbox{if } \gamma >1/2\\
 \frac{3}{2} &\hbox{if } \gamma \leq 1/2 .
 \end{cases}
$$
Since we are assuming $\sigma+2>0$, we need
$$
\gamma<\frac{1}{2}+\frac{3}{4}\beta,
$$
which is slightly stronger than (\ref{e:1.7}).

\bigskip

{\bf Acknowledgement.} We thank T. Kumagai for  discussions on topics related to deterministic fractional time equations.

\end{document}